\pdfoutput=1
\documentclass[a4paper,reqno]{amsart}
\usepackage[utf8]{inputenc}
\usepackage{color}
\usepackage{prettyref}
\usepackage{mathrsfs}
\usepackage{mathtools}
\usepackage{amstext}
\usepackage{amsthm}
\usepackage{amssymb}
\usepackage{stmaryrd}
\usepackage{microtype}
\usepackage[bookmarks=false,
 breaklinks=true,pdfborder={0 0 0},pdfborderstyle={},backref=false,colorlinks=true]
 {hyperref}
\hypersetup{
 pdfstartview=FitV,linkcolor=black,citecolor=black}

\makeatletter

\pdfpageheight\paperheight
\pdfpagewidth\paperwidth
\numberwithin{equation}{section}
\numberwithin{figure}{section}
\theoremstyle{plain}
\newtheorem{thm}{\protect\theoremname}[section]
\theoremstyle{plain}
\newtheorem{prop}[thm]{\protect\propositionname}
\theoremstyle{remark}
\newtheorem{rem}[thm]{\protect\remarkname}
\theoremstyle{plain}
\newtheorem{lem}[thm]{\protect\lemmaname}
\theoremstyle{definition}
\newtheorem{defn}[thm]{\protect\definitionname}
\theoremstyle{definition}
\newtheorem{example}[thm]{\protect\examplename}
\newenvironment{lyxlist}[1]
	{\begin{list}{}
		{\settowidth{\labelwidth}{#1}
		 \setlength{\leftmargin}{\labelwidth}
		 \addtolength{\leftmargin}{\labelsep}
		 }}
	{\end{list}}

\usepackage{fullpage}
\usepackage{bbm}
\usepackage{pgfplots}
\usepgfplotslibrary{colormaps,fillbetween}
\pgfplotsset{compat=1.15}
\usepackage{mathrsfs}
\usetikzlibrary{arrows}

\usepackage{txfonts}
\usepackage{textcomp}
\usepackage{xcolor}
 
\usepackage{dsfont}
\usepackage{graphicx}
 \usepackage{caption}
\usepackage{subcaption}
 
\usepackage{float}
\usepackage{tikz}
\usepackage[foot]{amsaddr}
\usepackage{enumitem}

\renewcommand{\d}{\mathrm{d}}
\renewcommand{\rho}{\varrho}
\renewcommand{\epsilon}{\varepsilon}

\newcommand{\1}{\mathbbm{1}} 
\newcommand{\e}{\mathrm{e}} 
\newcommand{\er}{\mathfrak{e}} 
\renewcommand{\c}{\mathfrak{C}}
\newcommand{\D}{\mathcal{D}}

\DeclareMathOperator{\supp}{supp}
\DeclareMathOperator{\card}{card}

\renewcommand{\tilde}{\widetilde}
\renewcommand{\emptyset}{\varnothing}
\renewcommand{\phi}{\varphi}
\def\N{\mathbb N}
\def\d{\;\mathrm d}
\def\R{\mathbb{R}}
\def\Z{\mathbb{Z}}

\def\Q{\mathcal{Q}}

\def\J{\mathfrak{J}}
\def\GL{\tau}
\definecolor{lime}{HTML}{A6CE39}
\DeclareRobustCommand{\orcidicon}{%
	\begin{tikzpicture}
	\draw[lime, fill=lime] (0,0) 
	circle [radius=0.16] 
	node[white] {{\fontfamily{qag}\selectfont \tiny ID}};
	\draw[white, fill=white] (-0.0625,0.095) 
	circle [radius=0.007];
	\end{tikzpicture}
	\hspace{-2mm}
}

\foreach \x in {A, ..., Z}{%
	\expandafter\xdef\csname orcid\x\endcsname{\noexpand\href{https://orcid.org/\csname orcidauthor\x\endcsname}{\noexpand\orcidicon}}
}

\AtBeginDocument{
\DeclareFieldFormat[article]{title}{{#1}}
\DeclareFieldFormat[unpublished]{title}{{#1}}
\DeclareFieldFormat[proceedings]{title}{{#1}}
\DeclareFieldFormat[incollection]{title}{{#1}}
\DeclareFieldFormat[unpublished]{title}{{\em #1}}
\renewbibmacro{in:}{}
\DeclareFieldFormat[article]{pages}{#1}
}

\newrefformat{prop}{Proposition~\ref{#1}}
\newrefformat{fact}{Fact~\ref{#1}}
\newrefformat{exa}{Example~\ref{#1}}
\newrefformat{cor}{Corollary~\ref{#1}}
\newrefformat{assu}{Assumption~\ref{#1}}
\newrefformat{subsec}{Section~\ref{#1}}
\newrefformat{sec}{Section~\ref{#1}}
\newrefformat{def}{Definition~\ref{#1}}
\newrefformat{rem}{Remark~\ref{#1}}
\newrefformat{enu}{(\ref{#1})}

 \makeatother

\usepackage[style=authoryear,backend=bibtex,style=alphabetic,  backref=false, giveninits=true,citestyle=alphabetic, natbib=false, url=false, isbn=false,doi=true,eprint=false,maxbibnames=99]{biblatex}
\providecommand{\definitionname}{Definition}
\providecommand{\examplename}{Example}
\providecommand{\lemmaname}{Lemma}
\providecommand{\propositionname}{Proposition}
\providecommand{\remarkname}{Remark}
\providecommand{\theoremname}{Theorem}

\begin{document}
\title[Quantization dimensions of negative order]{ Quantization dimensions of negative order}
\author{Marc Kesseböhmer\orcidA{}}
\email{mhk@uni-bremen.de}
\author{Aljoscha Niemann\orcidB{}}
\email{niemann1@uni-bremen.de}
\address{Institute for Dynamical Systems, FB 3 – Mathematics and Computer Science,
University of Bremen, Bibliothekstr. 5, 28359 Bremen, Germany}
\begin{abstract}
We investigate the possibility of defining meaningful upper and lower
quantization dimensions for a compactly supported Borel probability
measure of order $r$, including negative values of $r$. To this
end, we use the concept of partition functions, which generalizes
the idea of the $L^{q}$-spectrum and in this way naturally extends
the work in {[}M. Kesseböhmer, A. Niemann, and S. Zhu. Quantization
dimensions of probability measures via Rényi dimensions. \emph{Trans.
Amer. Math. Soc. }376.7 (2023){]}. In particular, we provide natural
fractal geometric bounds as well as easily verifiable necessary conditions
for the existence of the quantization dimensions. The exact asymptotics
of the quantization error of negative order for absolutely continuous
measures are stated, whereby an open question from {[}S. Graf, H.
Luschgy. \emph{Math. Proc. Cambridge Philos. Soc.} 136, 3 (2004){]}
regarding the geometric mean error is also answered in the affirmative.
\end{abstract}


\maketitle

\section{Introduction and statement of main results\label{sec:Introduction-and-background}}

The quantization problem for probability measures dates back to the
1980s (cf. \cite{MR651809}), where it was first established in the
context of information theory, and has recently received renewed attention,
as appropriate quantization of continuous data is fundamental to many
machine learning applications \cite{MR4603662,MR4119154}.

The goal is to examine the asymptotic behavior of the errors in the
convergence of a sequence of approximations of a given random variable
with a quantized version of that random variable (i.\,e\@., with
a random variable that takes at most $n\in\N$ different values) in
terms of $r$-th power mean with $r\geq0$. The quantization dimension
(of order $r$) is then defined as the exponential rate of this convergence
as $n$ tends to infinity. These ideas have been extensively explored
in the mathematical literature by numerous authors, including \cite{MR1434978,MR1844394,MR1764176,MR2055056,MR3035159,MR3300877,MR3345508,MR3558153,pages2015introduction,MR3415729,MR3489058,MR3636381,MR3731698,MR4119544,MR4206065,KN22b,MR4704055,zhu2024asymptotic}.
The objective of this study is to demonstrate that this concept can
be naturally extended to encompass also negative values of $r$. In
turn, we gain new insights into the asymptotics of the geometric mean
error for absolutely continuous measures, for which only the upper
bound has been established in \cite{MR2055056}. In this paper, Graf
and Luschgy asked whether the upper bound they established already
represents the exact asymptotic for any absolute continuous measure,
a question we can now answer in the affirmative under very mild conditions
(see \prettyref{thm:absoltutelyContinuousCase} and the following
remark).

\subsection{Basic set-up and first observations}

Let $X$ be a bounded random variable with values in the normed vector
space $\left(\R^{d},\left\Vert \,\cdot\,\right\Vert \right)$, $d\in\N$,
defined on the probability space $\left(\Omega,\mathcal{A},\mathbb{P}\right)$
and we let $\nu\coloneqq\mathbb{P}\circ X^{-1}$ denote its (compactly
supported) distribution. For a given $n\in\N$, let $\mathcal{F}_{n}$
denote the set of all Borel measurable functions $f:\R^{d}\rightarrow\R^{d}$
which take at most $n$ different values, i.\,e\@. with $\card\left(f\left(\R^{d}\right)\right)\leq n$,
and call an element of $\mathcal{F}_{n}$ an \emph{$n$-quantizer}.
Our aim is to approximate $X$ with a quantized version of $X$, i.\,e\@.
$X$ will be approximated by $f\circ X$ with $f\in\mathcal{F}_{n}$
where we quantify this approximation with respect to the $r$-quasi-norm.
More precisely, we are interested in the\emph{ $n$-th quantization
error} of $\nu$\emph{ of order $r\in\left(0,+\infty\right]$} given
by
\[
\er_{n,r}(\nu)\coloneqq\inf_{f\in\mathcal{F}_{n}}\left\Vert X-f\circ X\right\Vert _{L_{\mathbb{P}}^{r}}=\inf_{f\in\mathcal{F}_{n}}\left\Vert \text{id}-f\right\Vert _{L_{\nu}^{r}},
\]
where $\left\Vert g\right\Vert _{L_{\mathbb{P}}^{r}}\coloneqq\left(\int\left\Vert g\right\Vert ^{r}\,\text{d}\mathbb{P}\right)^{1/r}$
for $0<r<\infty$ and $\left\Vert g\right\Vert _{L_{\mathbb{P}}^{\infty}}\coloneqq\inf\left\{ c\geq0:\left\Vert g\right\Vert \leq c\,\ensuremath{\mathbb{P}\text{-a.s.}}\right\} $
and $\text{id}:x\mapsto x$ denotes the identity map on $\R^{d}$.

The problem can also be expressed using only the distribution $\nu$,
which we assume is compactly supported throughout the paper. For every
$n\in\mathbb{N}$, we write for the set of non-empty sets with at
most $n$ elements $\mathcal{A}_{n}:=\left\{ A\subset\mathbb{R}^{d}:1\leq\card\left(A\right)\leq n\right\} $.
Then due to \cite[Lemma 3.1]{MR1764176} an equivalent formulation
of the $n$-th quantization error of $\nu$ of order $r\in\left[0,+\infty\right)$
is given by 
\[
\er_{n,r}\left(\nu\right)=\left\{ \begin{array}{ll}
\inf_{A\in\mathcal{A}_{n}}\left\Vert d\left(\,\cdot\,,A\right)\right\Vert _{L_{\nu}^{r}}, & r>0,\\
\inf_{A\in\mathcal{A}_{n}}\exp\int\log d(x,A)\d\nu(x), & r=0,
\end{array}\right.
\]
with $d(x,A)\coloneqq\min_{y\in A}\left\Vert x-y\right\Vert $. By
\cite[Lemma 6.1]{MR1764176} it follows that $\er_{n,r}\left(\nu\right)\rightarrow0$
or more precisely, $\er_{n,r}\left(\nu\right)=O(n^{-1/d})$ and, if
$\nu$ is singular with respect to the Lebesgue measure, then $\er_{n,r}\left(\nu\right)=o\left(n^{-1/d}\right)$
(see \cite[Theorem 6.2]{MR1764176}). We define the \emph{upper }and\emph{
lower quantization dimension for $\nu$ of order $r$} by 
\[
\overline{D}_{r}\left(\nu\right):=\limsup_{n\to\infty}\frac{\log n}{-\log\er_{n,r}\left(\nu\right)},\;\;\underline{D}_{r}\left(\nu\right):=\liminf_{n\to\infty}\frac{\log n}{-\log\er_{n,r}\left(\nu\right)},
\]
with the natural convention that $1/\log(0)=0$. In particular, when
$\er_{n,r}\left(\nu\right)=0$ for all $n$ large, then the dimension
vanishes. If $\overline{D}_{r}\left(\nu\right)=\underline{D}_{r}\left(\nu\right)$,
we call the common value the \emph{quantization dimension for $\nu$
of order $r$} and denote it by $D_{r}\left(\nu\right)$.

We point out that for $r\geq1$ the relevant $r$-quasi-norm is indeed
a norm and only for $r\in\left(0,1\right)$ a quasi-norm, which means
that only a generalized triangular inequality holds. In this note,
we will investigate the possibility and significance of also considering
negative values for $r$ in the definition of the quantization error,
i.\,e\@. we extend the definition of $\er_{n,r}\left(\nu\right)$
to values $r<0$, by setting 
\[
\er_{n,r}\left(\nu\right)\coloneqq\inf_{A\in\mathcal{A}_{n}}\left(\int d\left(x,A\right)^{r}\d\nu\left(x\right)\right)^{1/r}.
\]

Our first observation is that if $\nu$ has an atom, that is there
exists $x\in\R^{d}$ with $\nu\left(\left\{ x\right\} \right)>0$,
then for all $r<0$ we have that the integrant in the definition of
$\er_{n,r}\left(\nu\right)$ is equal to $+\infty$ on a set of positive
measure whenever $x\in A$. Therefore, in this situation $\er_{n,r}\left(\nu\right)=0$
for all $n\in\N$ and hence $D_{r}\left(\nu\right)=0$ for $r<0$.
This means that if the measure has a too high concentration of mass,
the quantization dimension will vanish for negative values of $r$.
This idea can be taken further: We will assume without loss of generality
that the support of $\text{\ensuremath{\nu}}$ is contained in the
half-open unit cube $\Q\coloneqq(0,1]^{d}$, see \cite[Introduction]{KN22b}.
Although it is not strictly necessary to assume that $\nu$ is normalised,
it will always be assumed to be finite throughout the paper. Let us
define the $\infty$-dimension of $\nu$ by \textbf{
\[
\dim_{\infty}\left(\nu\right)\coloneqq\liminf_{n\to\infty}\frac{\max_{Q\in\mathcal{D}_{n}}\log\nu\left(Q\right)}{\log2^{-n}}=\liminf_{r\to0}\frac{\sup_{x\in\R^{d}}\log\nu\left(B_{r}\left(x\right)\right)}{\log r},
\]
}where $B_{r}\left(x\right)$ denotes the ball in $\left(\R^{d},\left\Vert \,\cdot\,\right\Vert \right)$
of radius $r>0$ and center $x$, and $\mathcal{D}_{n}$ denotes the
partition of $\Q$ by half-open cubes of the form $\prod_{i=1}^{d}\left(k_{i}2^{-n},\left(k_{i}+1\right)2^{-n}\right]$
with $\left(k_{1},\ldots,k_{d}\right)\in\Z^{d}$. We set $\mathcal{D}\coloneqq\bigcup_{n\in\N}\mathcal{D}_{n}$,
which defines –after adding the empty set– a semiring of sets. Note
that $\dim_{\infty}\left(\nu\right)\leq d$. The following observation
provides us with the relevant range of meaningful values for the order
of quantization.
\begin{prop}
\label{prop:r<-dimoo(nu)}For a probability measure $\nu$ on $\R^{d}$
and $r<-\dim_{\infty}\left(\nu\right)$, we have
\[
\er_{n,r}\left(\nu\right)=0,\;\text{for all}\:n\in\N
\]
and, in particular, $D_{r}\left(\nu\right)=0$. Furthermore, if the
measure $\nu$ has an absolutely continuous part with respect to the
Lebesgue measure, then the above identity holds for $r=-d$.
\end{prop}

\begin{rem}
Note that for absolutely continuous measure the second claim in the
above proposition is only meaningful if $d=\dim_{\infty}\left(\nu\right)$.
In instances where the quantity $\dim_{\infty}\left(\nu\right)$ is
strictly smaller than $d$ and $\nu$ is absolutely continuous, we
refer to Examples \ref{exa:abs_cnt_with_strict_inequ} and \ref{exa:Our-second-example}.

Also note that \prettyref{prop:r<-dimoo(nu)} is consistent with our
first observation, since $\dim_{\infty}\left(\nu\right)=0$ whenever
$\nu$ has atoms. Hence the interesting range of values for $r$ in
the formulation of the quantization problem is the interval $\left[-\dim_{\infty}\left(\nu\right),+\infty\right]$
which we will focus on in this paper. This observation also shows
that the quantization dimension with negative order is particularly
sensitive to regions of high concentration of the underlying measure
— a fact that could prove particularly useful for applications.
\end{rem}

\begin{proof}
Fix $r<-\dim_{\infty}\left(\nu\right)$ and $t\in\left(r,-\dim_{\infty}\left(\nu\right)\right)$.
By definition of $\dim_{\infty}\left(\nu\right)$, we find sequences
$\left(s_{\ell}\right)\in\left(\R_{>0}\right)^{\N}$ and $\left(x_{\ell}\right)\in\left(\R^{d}\right)^{\N}$
such that $s_{\ell}\searrow0$ and $\nu\left(B_{s_{\ell}}\left(x_{\ell}\right)\right)\geq s_{\ell}^{-t}$.
Setting $V_{n,r}\left(\nu\right)\coloneqq\er_{n,r}\left(\nu\right)^{r}$,
this gives, for $\ell\to\infty$, 
\begin{align*}
V_{n,r}\left(\nu\right) & \geq V_{1,r}\left(\nu\right)\geq\sup_{A\in\mathcal{A}_{1}}\int d\left(x,A\right)^{r}\d\nu\left(x\right)\\
 & \geq\int_{B_{s_{\ell}}\left(x_{\ell}\right)}d\left(x,\left\{ x_{\ell}\right\} \right)^{r}\d\nu\left(x\right)\geq\nu\left(B_{s_{\ell}}\left(x_{\ell}\right)\right)s_{\ell}^{r}\geq s_{\ell}^{r-t}\to\infty.
\end{align*}
This proves the first claim.

To see that $\er_{n,-d}\left(\nu\right)=0$ for each $n\in\N$, by
Lebesgue's differentiation theorem we find a point $x_{0}\in\R^{d}$
such that $\lim_{k\to\infty}\nu\left(B_{2^{-k}}\left(x_{0}\right)\right)/2^{-kd}=c>0$.
Consequently, for $k\in\N$ large enough (say $k\geq k_{0})$ we have
for $C_{k}\coloneqq B_{2^{-k}}\left(x_{0}\right)\setminus B_{2^{-k-1}}\left(x_{0}\right)$
that $\nu\left(C_{k}\right)>\left(c/2\right)2^{-d\left(k+1\right)}.$
This gives for each $n\in\N$ 
\[
V_{n,-d}\left(\nu\right)\geq\sum_{k\geq k_{0}}\int_{C_{k}}d\left(x,\left\{ x_{0}\right\} \right)^{-d}\d\nu\left(x\right)\geq\sum_{k\geq k_{0}}\frac{c}{2}2^{-d\left(k+1\right)}2^{dk}=\infty.
\]
This shows the second claim.
\end{proof}

\subsection{\label{subsec:Absolutely-continuous-measures}Absolutely continuous
measures – exact asymptotics}

We are able to extend the classical result in quantization theory
for absolutely continuous probability measures with respect to the
Lebesgue measure, which goes back to Zador \cite{MR651809} and was
generalized by Bucklew and Wise in \cite{MR651819}; for a rigorous
proof for $r\geq1$, which also works for $r>0$, see \cite[Theorem 6.2]{MR1764176}.

If for $\kappa>0$ the limit 
\[
\c_{r,\kappa}\left(\nu\right)\coloneqq\lim_{n\to\infty}n^{1/\kappa}\er_{n,r}\left(\nu\right)
\]
exists in $\left[0,+\infty\right]$, we refer to its value as the
$\kappa$\emph{-dimensional quantization coefficient of order $r$
of $\nu$}. In the literature, the lower and upper quantization coefficients
defined via limes inferior and superior have also frequently been
considered \cites{MR2055056}{MR1844394}{MR2990259}{MR3636381}{MR4616041}.
If this number (or more generally, its upper and lower version) is
positive and finite, the quantization dimension of order $r$ also
exists and is equal to $\kappa$. Note that $\c_{r,\kappa}\left(\nu\right)=0$
is equivalent to saying that $\er_{n,r}\left(\nu\right)=o\left(n^{-1/\kappa}\right)$.

In \prettyref{subsec:Absolutely-continuous-case}, we will see that,
for the Lebesgue measure restricted to $\Q$ and for which we write
$\Lambda$, the coefficient $\c_{r,d}\left(\Lambda\right)$ also exists
and is finite for negative $r$, and positive only for $r>-d$ (cf.
\prettyref{lem:QC_UnitCube}). In general, for an absolutely continuous
probability $\nu\coloneqq h\Lambda$ with density $h$ we set $s_{h}\coloneqq\sup\left\{ s>0:\left\Vert h\right\Vert _{s}<\infty\right\} $,
where for abbreviation we write $\left\Vert \,\cdot\,\right\Vert _{s}$
instead of $\left\Vert \,\cdot\,\right\Vert _{L_{\Lambda}^{s}}$.
By \prettyref{lem:AbsolLq} and \ref{lem:lowerboudonBeta} we have
that $s_{h}\leq d/\left(d-\dim_{\infty}\left(\nu\right)\right)$.
In particular, $-d\leq-\dim_{\infty}\left(\nu\right)\leq d/s_{h}-d$
and these inequalities might be strict as \prettyref{exa:abs_cnt_with_strict_inequ}
demonstrates. For a measurable function $h$ we set
\[
\Phi_{r}\left(h\right)\coloneqq\begin{cases}
\left\Vert h\right\Vert _{d/\left(d+r\right)}^{1/r}, & r\in\left(-d,\infty\right)\setminus\left\{ 0\right\} ,\\
\exp\left(-\left(1/d\right)\int h\log\left(h\right)\d\Lambda\right), & r=0,\\
0, & r\leq-d
\end{cases}
\]
where $\left\Vert h\right\Vert _{d/\left(d+r\right)}^{1/r}=0$ whenever
$r<0$ and $h^{d/\left(d+r\right)}$ not integrable.
\begin{thm}
\emph{\label{thm:absoltutelyContinuousCase} }Let $\nu=h\Lambda$
be an absolutely continuous Borel probability measure on $\Q$ with
density $h$ and $s_{h}$ given as above. For $r\neq d/s_{h}-d$,
and also for $r=d/s_{h}-d$ provided $\left\Vert h\right\Vert _{s_{h}}=\infty$,
the $d$-dimensional quantization coefficient of order $r$ of $\nu$
exists and its finite value equals
\[
\c_{r,d}\left(\nu\right)=\c_{r,d}\left(\Lambda\right)\Phi_{r}\left(h\right).
\]
Moreover, we have $\c_{r,d}\left(\nu\right)>0$ for $r>d/s_{h}-d$,
and $\c_{r,d}\left(\nu\right)=0$ for $r<d/s_{h}-d$ or, if $\left\Vert h\right\Vert _{s_{h}}=\infty$,
also for $r=d/s_{h}-d$.
\end{thm}

The proof of this first theorem is outlined at the end of our paper
in \prettyref{subsec:Absolutely-continuous-case}. This section covers
the necessary prerequisites and explains the main steps of the proof.
\begin{rem}
The special value $r=0$, for which $\er_{n,0}\left(\nu\right)$ coincides
with the geometric mean error, has been previously examined in several
papers (see \cite{MR2055056,MR2416331,MR2727173,MR2954531,MR3016541,MR3035089,MR3300877,MR3415729,MR4206065}).
The present approach, which now takes negative orders into account,
provides a straightforward answer to a question posed by Graf and
Luschgy over 20 years ago. In \cite{MR2055056} they showed that 
\begin{equation}
\limsup_{n\to\infty}n^{1/d}\er_{n,0}\left(\nu\right)\leq\c_{0,d}\left(\Lambda\right)\Phi_{0}\left(h\right),\label{eq:uppBoundG-L}
\end{equation}
and pointed out that it ``\emph{remains an open question whether
one has a genuine limit in \prettyref{eq:uppBoundG-L} for absolutely
continuous probabilities $\mathbb{P}$ different from uniform distributions
on cubes}.'' Whenever $s_{h}>1$, our first main result covers the
geometric mean error, i.\,e\@. $r=0$, and hence answers the question
in the affirmative.

It should also be noted that explicit formulas for $\c_{r,d}\left(\Lambda\right)$
are often difficult to obtain and that its value dependents on the
underlying norm on $\R^{d}$. Nevertheless, the quantization dimension
is independent of the underlying norm. For the sake of simplicity,
both the euclidean and the maximum norm will be used in the subsequent
proofs.

In \prettyref{exa:Our-second-example} we provide an example with
$\left\Vert h\right\Vert _{s_{h}}<\infty$ such that the $d$-dimensional
quantization coefficient of order $d/s_{h}-d$ exists and is positive.
This critical situation $\left\Vert h\right\Vert _{s_{h}}<\infty$
will be revisited in \prettyref{rem:CriticlValue}.

The fact that $\left\Vert h\right\Vert _{d/\left(d+r\right)}=\infty$
for a certain negative $r$ indicates a strong concentration of mass
and can be considered analogous to singular measures and positive
$r$. In both cases, the convergence rate of the quantization error
is $o\left(n^{-1/d}\right)$.
\end{rem}

\subsection{Partition functions and related concepts}

Building upon the ideas developed in the context of spectral problems
for Kre\u{\i}n–Feller operators in \cite{KN2022,KN2022b} and for
approximation orders of Kolmogorov, Gel'fand, and linear widths in
\cite{KesseboehmerWiegmann,MR4444736}, we will address the quantization
problem as initiated in \cite{KN22b}, where we considered positive
order and determind the exact value of the upper quantization dimension
with the help of the $L^{q}$-spectrum. In this instance, however,
we will consider negative order, for which novel concepts and strategies
are required, particularly the more general concept of the $\J$-partition
function. Interestingly, the underlying methods elaborated in \cite{MR4882804}
and initially utilized in the context of spectral problems prove to
be precisely the appropriate tools in this context. The main difference
to the results obtained for the positive order case is that, this
time, our formalism provides the exact value of the lower quantization
dimension, as well as an upper bound for the upper quantization dimension
(see \prettyref{thm:Main}). Additionally, we provide the exact value
for the upper and lower quantisation dimension in the regular cases
(see \prettyref{subsec:Regularity-results}).

With $\mathcal{D}_{n}\left(Q\right)\coloneqq\left\{ Q'\in\mathcal{D}_{n}:Q'\subset Q\right\} $,
$n\in\N$, we let $\mathcal{D}\left(Q\right)\coloneqq\bigcup_{n\in\N}\mathcal{D}_{n}\left(Q\right)$.
For $r>-\dim_{\infty}\left(\nu\right)$, we then set
\[
\J\coloneqq\J_{\nu,r}:\mathcal{D}\to\R_{\geq0},\quad Q\mapsto\max_{Q'\in\mathcal{D}\left(Q\right)}\nu\left(Q'\right)\Lambda\left(Q'\right)^{r/d}
\]
and define the \emph{$\J$-partition function} $\tau_{\J}$, for $q\in\R_{\geq0}$,
by
\[
\tau_{\J}(q)\coloneqq\limsup_{n\rightarrow\infty}\tau_{\J,n}(q),\:\;\text{with}\quad\tau_{\J,n}\left(q\right)\coloneqq\frac{\log\left(\sum_{Q\in\mathcal{D}_{n},\J\left(Q\right)>0}\J(Q)^{q}\right)}{\log(2^{n})}.
\]
Define the critical value
\[
q_{r}\coloneqq\inf\left\{ q>0:\tau_{\J_{\nu,r}}\left(q\right)<0\right\} .
\]
Note that for $r\geq0$ we have $\J_{\nu,r}\left(Q\right)=\nu\left(Q\right)\Lambda\left(Q\right)^{r/d}$
and consequently for non-negative values of $r$, $\tau_{\J}\left(q\right)=\beta_{\nu}\left(q\right)-qr$,
where 
\[
\beta_{\nu}:q\mapsto\limsup_{n\rightarrow\infty}\beta_{\nu,n}\left(q\right)\quad\text{with }\quad\beta_{\nu,n}\left(q\right)\coloneqq\text{\ensuremath{\frac{\log\left(\sum_{Q\in\mathcal{D}_{n},\nu\left(Q\right)>0}\nu(Q)^{q}\right)}{\log(2^{n})}}}
\]
denotes the \emph{$L^{q}$-spectrum of} $\nu$, which was the central
object in \cite{KN22b}. In \prettyref{exa:Menger_Sponge} (see also
\prettyref{fig:Moment-generating-function}) we provide a measure
for which $\tau_{\J_{\nu,r}}\left(q\right)=\beta_{\nu}\left(q\right)-qr$
holds also for negative values of $r$. We always have $\tau_{\J_{\nu,r}}\left(q\right)\geq\beta_{\nu}\left(q\right)-qr$
by the definition $\tau_{\J}$. It is easy to construct purely atomic
measures such that $q_{r}=0$ for all $r>0$ and $\beta_{\nu}(0)>0$,
see \cite{KN2022}. In this case it turns out that the upper quantization
dimension is also $0$.

For the following fundamental inequalities for $q_{r}$ we assume
$\dim_{\infty}(\nu)>0$. Then $\beta_{\nu}(0)=\tau_{\J_{\nu,r}}\left(0\right)=\overline{\dim}_{M}\left(\nu\right)\geq\dim_{\infty}(\nu)>0$
and for $-\dim_{\infty}(\nu)<r<0$ and $q\geq0$, we have $\tau_{\J_{\nu,r}}\left(q\right)\geq\beta_{\nu}(q)-qr$.
For $q\geq1$, the convexity of $\tau_{\J_{\nu,r}}$ gives
\[
\tau_{\J_{\nu,r}}\left(q\right)\geq\left(\tau_{\J_{\nu,r}}\left(1\right)-\tau_{\J_{\nu,r}}\left(0\right)\right)q+\tau_{\J_{\nu,r}}\left(0\right)\geq\left(-r-\beta_{\nu}\left(0\right)\right)q+\beta_{\nu}\left(0\right).
\]
On the other hand, again the convexity of $\tau_{\J_{\nu,r}}$ gives
for $q\geq1$ 
\[
\tau_{\J_{\nu,r}}\left(q\right)\leq\left(-\dim_{\infty}(\nu)-r\right)q+\dim_{\infty}(\nu)+r+\tau_{\J_{\nu,r}}\left(1\right).
\]
Combining both inequalities and using the definition of $q_{r}$ proves
for $r\in\left(-\dim_{\infty}(\nu),0\right)$
\begin{equation}
1<\frac{\overline{\dim}_{M}\left(\nu\right)}{\overline{\dim}_{M}\left(\nu\right)+r}\leq q_{r}\leq1+\frac{\tau_{\J_{\nu,r}}\left(1\right)}{\dim_{\infty}(\nu)+r}\leq1+\frac{\overline{\dim}_{M}\left(\nu\right)-\dim_{\infty}(\nu)}{\dim_{\infty}(\nu)+r}.\label{eq:qr_exists-1}
\end{equation}
Moreover, for $r>0$, due the convexity of $\beta_{\nu}$ and the
fact that $\beta_{\nu}(q)=\tau_{\J_{\nu,r}}\left(q\right)+rq\leq\left(1-q\right)\overline{\dim}_{M}\left(\nu\right)$
for $q\in\left[0,1\right]$, we conclude 
\begin{equation}
0<q_{r}\leq\frac{\overline{\dim}_{M}\left(\nu\right)}{\overline{\dim}_{M}\left(\nu\right)+r}<1.\label{eq:qr_exists}
\end{equation}
Further, we will need some ideas from entropy theory: Let $\Pi$ denote
the set of finite partitions of $\Q$ by cubes from $\mathcal{D}$.
We define
\[
\mathcal{M}_{\nu,r}\left(x\right)\coloneqq\inf\left\{ \card\left(P\right):P\in\Pi,\,\max_{Q\in P}\J_{\nu,r}\left(Q\right)<1/x\right\} .
\]
and
\[
\overline{h}_{\nu,r}\coloneqq\limsup_{x\to\infty}\frac{\log\mathcal{M}_{\nu,r}\left(x\right)}{\log x},\quad\underline{h}_{\nu,r}\coloneqq\liminf_{x\to\infty}\frac{\log\mathcal{M}_{\nu,r}\left(x\right)}{\log x}
\]
will be called the\emph{ upper, }resp.\emph{ lower, $\left(\nu,r\right)$-partition
entropy.} We write $\overline{h}_{r}\coloneqq\overline{h}_{\nu,r}$
and\emph{ $\underline{h}_{r}\coloneqq\lim_{\varepsilon\downarrow0}\underline{h}_{\nu,r-\varepsilon}$}.

For all $n\in\N$ and $\alpha>0$, we define 
\[
\mathcal{N}_{\nu,\alpha,r}\left(n\right)\coloneqq\card N_{\nu,\alpha,r}\left(n\right),\quad N_{\nu,\alpha,r}\left(n\right)\coloneqq\left\{ Q\in\mathcal{D}_{n}:\J_{\nu,r}\left(Q\right)\geq2^{-\alpha n}\right\} ,
\]
and set
\[
\overline{F}_{\nu,r}\left(\alpha\right)\coloneqq\limsup_{n}\frac{\log^{+}\left(\mathcal{N}_{\nu,\alpha,r}\left(n\right)\right)}{\log2^{n}}\text{ and }\underline{F}_{\nu,r}\left(\alpha\right)\coloneqq\liminf_{n}\frac{\log^{+}\left(\mathcal{N}_{\nu,\alpha,r}\left(n\right)\right)}{\log2^{n}}.
\]
Following \cite{KN2022b}, we refer to the quantities
\[
\overline{F}_{r}\coloneqq\overline{F}_{\nu,r}\coloneqq\sup_{\alpha>0}\frac{\overline{F}_{\nu,r}\left(\alpha\right)}{\alpha}\text{ and }\underline{F}_{r}\coloneqq\underline{F}_{\nu,r}\coloneqq\sup_{\alpha>0}\frac{\underline{F}_{\nu,r}\left(\alpha\right)}{\alpha}
\]
as the $\left(\nu,r\right)$\emph{-upper}, resp. \emph{lower, optimized
coarse multifractal dimension}. We know from general consideration
obtained in \cite{MR4882804} that we always have $\overline{F}_{r}=q_{r}=\overline{h}_{r}\geq\underline{h}_{r}\geq\underline{F}_{r}$,
see \prettyref{prop:GeneralResultOnPartitionFunction}.

\subsection{Main results}

In our main theorem we combine the main result in \cite{KN22b} for
positive $r$ with our new results on negative $r$. Note that our
assumption in \cite{KN22b}, that \emph{$\sup_{x\in(0,1)}\beta_{\nu}(x)>0$,}
is implied by our stronger assumption $\dim_{\infty}\left(\nu\right)>0$.
It is important to note that under this stronger assumption $r>0$
implies $0<q_{r}<1$, just as $-\dim_{\infty}\left(\nu\right)<r<0$
implies $q_{r}>1$ (see \prettyref{eq:qr_exists} and \prettyref{eq:qr_exists-1}).
Finally, we set 
\[
\dim_{H}\left(\nu\right)\coloneqq\inf\left\{ \dim_{H}\left(A\right):\nu\left(A\right)>0\right\} 
\]
with $\dim_{H}\left(A\right)$ referring to the \emph{Hausdorff dimension}
of $A\subset\R^{d}$. The proof of our main theorem is postponed to
the last section.
\begin{thm}
\label{thm:Main}For a compactly supported probability measure $\nu$
on $\R^{d}$ we have
\begin{enumerate}
\item \label{enu:1}for $r\in\left(0,+\infty\right)$, \emph{
\[
\frac{r\underline{F}_{r}}{1-\underline{F}_{r}}\leq\underline{D}_{r}\left(\nu\right)\leq\frac{r\underline{h}_{r}}{1-\underline{h}_{r}}\leq\frac{r\overline{h}_{r}}{1-\overline{h}_{r}}=\overline{D}_{r}\left(\nu\right)=\frac{rq_{r}}{1-q_{r}}=\frac{r\overline{F}_{r}}{1-\overline{F}_{r}},
\]
}
\item \label{enu:2}for $r\in\left(-\dim_{\infty}\left(\nu\right),0\right)$,
\begin{enumerate}
\item ${\displaystyle \frac{r\overline{F}_{r}}{1-\overline{F}_{r}}=\frac{rq_{r}}{1-q_{r}}=\underline{D}_{r}\left(\nu\right)=\frac{r\overline{h}_{r}}{1-\overline{h}_{r}}},$
\item ${\displaystyle \frac{r\overline{h}_{r}}{1-\overline{h}_{r}}\leq\frac{r\underline{h}_{r}}{1-\underline{h}_{r}}\leq\overline{D}_{r}\left(\nu\right)\leq\frac{r\underline{F}_{r}}{1-\underline{F}_{r}}}$
under the assumption $\underline{F}_{r}>1$,
\end{enumerate}
\item \label{enu:3} for $r=0$, and under the assumption that $\dim_{\infty}\left(\nu\right)>0$
and $r\mapsto q_{r}$ is differentiable at $0$, we have $\beta_{\nu}$
is differentiable at $1$ and 
\[
D_{0}\left(\nu\right)=-\beta_{\nu}'\left(1\right)=\dim_{H}\left(\nu\right).
\]
\end{enumerate}
\end{thm}

\begin{rem}
It should be noted that in \cite{KN22b}, the notion of the \emph{(upper)
generalized Rényi dimension of $\nu$ }\cite{zbMATH03124309} was
employed. In our context, it is necessary to replace the concept of
the $L^{q}$-spectrum with the more general concept of the $\tau_{\J_{\nu,r}}$-partition
function, where for $r>-\dim_{\infty}\left(\nu\right)$, we set
\[
\overline{\mathfrak{R}}_{\nu}\left(q\right)\coloneqq\begin{cases}
\tau_{\J_{\nu,r}}(q)/\left(1-q\right), & \text{for }q\neq1,\\
\limsup_{n}\left(\sum_{C\in\mathcal{D}_{n}}\nu(C)\log\nu(C)\right)/\log\left(2^{-n}\right), & \text{for }q=1.
\end{cases}
\]
In this manner, we derive for both positive and negative values of
$r$ the remarkable identity 
\[
\overline{\mathfrak{R}}_{\nu}\left(q_{r}\right)=\frac{rq_{r}}{1-q_{r}},
\]
which can be applied to \prettyref{thm:Main} \prettyref{enu:1} and
\prettyref{enu:2}, providing an alternative expression for $\overline{D}_{r}\left(\nu\right)$
whenever $r$ is positive and $\underline{D}_{r}\left(\nu\right)$
whenever $r$ is negative.

The connection between the derivatives of $r\mapsto q_{r}$ and $\beta_{\nu}$
becomes apparent in the proof at the beginning of \prettyref{sec:Upper-bounds-1},
where we show that $\beta'_{\nu}$ is differentiable at $1$ with
$\beta'_{\nu}\left(1\right)=1/\partial_{r}q_{r}|_{r=0}$. This fact
is contingent on the monotonicity of $r\mapsto\underline{D}_{r}\left(\nu\right)$
on $\left(-\dim_{\infty}\left(\nu\right),0\right]$, which we establish
in \prettyref{lem:monontonicity}.
\end{rem}

\subsection{Fractal-geometric bounds and basic properties}

For a compactly supported probability measure $\nu$ on $\R^{d}$,
we let let $\overline{\dim}_{M}\left(\nu\right)\coloneqq\overline{\dim}_{M}\left(\supp\left(\nu\right)\right)$,
where $\overline{\dim}_{M}(A)$, denotes the \emph{upper Minkowski
dimension} of the bounded set $A\subset\R^{d}$.

The fact that $\tau_{\J}$ is proper, convex with $\beta_{\nu}\left(0\right)=\tau_{\J}\left(0\right)=\overline{\dim}_{M}\left(\nu\right)$,
$\tau_{\J}\left(q\right)\geq\beta_{\nu}\left(q\right)-rq$ for $q\geq0$,
and that its asymptotic slope $\lim_{q\to\infty}\tau_{\J}\left(q\right)/q$
is equal to $-r-\dim_{\infty}\left(\nu\right)$ plays a crucial role
for the following observations. To start this investigation, we first
make a simple observation which is well known for positive $r$.
\begin{lem}
\label{lem:monontonicity}If $\dim_{\infty}\left(\nu\right)>0$, then
the functions $r\mapsto\underline{D}_{r}\left(\nu\right)$ and $r\mapsto\overline{D}_{r}\left(\nu\right)$
are both monotonically increasing on the interval $\left(-\dim_{\infty}\left(\nu\right),0\right]$.
\end{lem}

\begin{proof}
Note that for $-\dim_{\infty}\left(\nu\right)<s<r<0$ and $A\in\mathcal{A}_{n}$
we have by Hölder's inequality (assuming without loss of generality
that $\nu$ is normalised) 
\[
\int d\left(x,A\right)^{r}\d\nu\leq\left(\int d\left(x,A\right)^{s}\d\nu\right)^{r/s}\leq\left(V_{n,s}\left(\nu\right)\right)^{r/s}\implies\left(\int d\left(x,A\right)^{r}\d\nu\right)^{1/r}\geq\er_{n,s}\left(\nu\right)
\]
and by Jensen's inequality, to cover the case $r=0$, 
\begin{align*}
\int\log d\left(x,A\right)^{r}\d\nu & \leq\log\int d\left(x,A\right)^{r}\d\nu\\
\implies & \exp\int\log d\left(x,A\right)\d\nu\geq\left(\int d\left(x,A\right)^{r}\d\nu\right)^{1/r}\geq\er_{n,r}\left(\nu\right).
\end{align*}
Taking in both cases the infimum over all $A\in\mathcal{A}_{n}$,
gives 
\[
\er_{n,s}\left(\nu\right)\leq\er_{n,r}\left(\nu\right)\leq\er_{n,0}\left(\nu\right)
\]
and the claim follows.
\end{proof}
Let us now discuss the behaviour of the quantization dimension at
the relevant boundary points $-\dim_{\infty}\left(\nu\right)$ and
$+\infty$. The case $+\infty$ follows from \cite{KN22b} and by
observing 
\begin{equation}
\lim_{r\nearrow+\infty}\overline{D}_{r}\left(\nu\right)=\lim_{q\searrow0}\beta_{\nu}\left(q\right)\leq\beta_{\nu}\left(0\right)=\overline{\dim}_{M}\left(\nu\right)=\overline{D}_{\infty}\left(\nu\right),\label{eq:UpperboundMinkowski}
\end{equation}
where the last equality can be found in \cite[Theorem 11.7, Proposition 11.9]{MR1764176}.
To discuss the boundary point $-\dim_{\infty}\left(\nu\right)$ we
need the following auxiliary quantity
\[
a_{\nu}\coloneqq\sup\left\{ q_{r}:r\in\left(-\dim_{\infty}\left(\nu\right),0\right)\right\} .
\]

\begin{lem}
\label{lem:GeometricBounds}For a compactly supported probability
measure $\nu$ on $\R^{d}$ with $\dim_{\infty}\left(\nu\right)>0$,
\begin{equation}
\dim_{\infty}\left(\nu\right)\leq\lim_{r\searrow-\dim_{\infty}\left(\nu\right)}\underline{D}_{r}\left(\nu\right)={\displaystyle \frac{a_{\nu}}{a_{\nu}-1}\dim_{\infty}\left(\nu\right)},\label{eq:boundaries_D_r}
\end{equation}
where we set $a_{\nu}/\left(a_{\nu}-1\right)=1$ whenever $a_{\nu}=+\infty$.
\end{lem}

The proofs will be given at the end of \prettyref{subsec:Lower-bounds},
where we also show that $a_{\nu}$ is strictly greater than $1$.
See also \cite[Proposition 1.7]{KN22b} for the corresponding upper
bounds for the upper quantization dimension of order $r>0$.
\begin{rem}
\label{rem:CriticlValue} It is easy to construct an absolutely continuous
measure with a density that has one singularity, such that $\beta_{\nu}$
is piecewise linear, which gives rise to the above described behavior
with $a_{\nu}<\infty$, see \prettyref{exa:Our-second-example}.

The above observation and lemma demonstrate that the boundary behavior
is only partially encoded by $\tau_{\J}$. Equation \prettyref{eq:UpperboundMinkowski}
shows that continuity of $r\mapsto\overline{D}_{r}\left(\nu\right)$
in $r=+\infty$ if and only if the $L^{q}$-spectrum $\beta_{\nu}$
is continuous from the right in $0$ or, equivalently, the same holds
for $\tau_{\J_{\nu,r}}$ for all $r>0$. In general, continuity of
$r\mapsto\underline{D}_{r}\left(\nu\right)$ in $r=-\dim_{\infty}\left(\nu\right)$
from the right cannot be derived from $\tau_{\J}$ alone. In point
of fact, \prettyref{prop:r<-dimoo(nu)} together with \prettyref{thm:Main}
covers all values of $r$ except the critical value $r=-\dim_{\infty}\left(\nu\right)$.
We would like to point out that the behaviour at this critical value
is not easily accessible and depends very much on the measure under
consideration:

For instance, as we have already observed in \prettyref{thm:absoltutelyContinuousCase},
for the uniform distribution $\Lambda$ we have $\dim_{\infty}\left(\Lambda\right)=d$,
while 
\[
D_{r}\left(\Lambda\right)=\begin{cases}
d, & r\in\left(-\dim_{\infty}\left(\Lambda\right),\infty\right),\\
0, & r\leq-\dim_{\infty}\left(\Lambda\right).
\end{cases}
\]
That is to say, $r\mapsto D_{r}\left(\Lambda\right)$ is discontinuous
and continuous from the left in $r=-\dim_{\infty}\left(\Lambda\right)$.

On the other hand, in \prettyref{exa:Our-second-example} we construct
an absolutely continuous measure $\nu$ with $d>\dim_{\infty}\left(\nu\right)>0$
such that 
\[
\lim_{r\searrow-\dim_{\infty}\left(\nu\right)}\underline{D}_{r}\left(\nu\right)=\underline{D}_{-\dim_{\infty}\left(\nu\right)}\left(\nu\right)=d,
\]
that is, $r\mapsto D_{r}\left(\nu\right)$ is this time discontinuous
and continuous from the right in $r=-\dim_{\infty}\left(\nu\right)$.
\end{rem}

\subsection{Regularity results\label{subsec:Regularity-results}}

As a second main result we find necessary conditions for the upper
and lower quantization dimension to coincide, which are easy to check
in many situations.
\begin{defn}
We define two notions of regularity for a compactly supported probability
measure $\nu$ on $\R^{d}$ such that $\dim_{\infty}\left(\nu\right)>0$.
\begin{enumerate}
\item The measure $\nu$ is called \emph{multifractal-regular at $r$ ($r$-MF-regular)}
if $\underline{F}_{\nu,r}=\overline{F}_{\nu,r}$.
\item The measure $\nu$ is called \emph{partition function regular at $r$}
\emph{($r$-PF-regular)}
\begin{enumerate}
\item $\tau_{\J}\left(q\right)=\liminf_{n}\tau_{\J,n}\left(q\right)\in\R$
for all $q\in\left(q_{r}-\varepsilon,q_{r}\right)$, for some $\varepsilon>0$,
or\label{enu:-LqRegularity1}
\item $\tau_{\J}\left(q_{r}\right)=\liminf_{n}\tau_{\J,n}\left(q_{r}\right)$
and $\tau_{\J}$ is differentiable at $q_{r}$.
\end{enumerate}
\end{enumerate}
\end{defn}

The following theorem, which is a direct consequence of \cite[Theorem 1.12]{MR4882804},
shows that the spectral partition function is a valuable auxiliary
concept to determine the quantization dimension for a given measure
$\nu$.
\begin{thm}
\label{thm:LqRegularImpliesRegular} The following regularity implications
hold for $r\in\left(-\dim_{\infty}\left(\nu\right),+\infty\right)\setminus\left\{ 0\right\} $:
\[
\nu\,\,\text{is }r\text{-PF-regular}\:\ensuremath{\implies}\:\nu\,\,\text{is \ensuremath{r}-MF-regular \, \ensuremath{\implies}\, }\underline{D}_{r}\left(\nu\right)=\overline{D}_{r}\left(\nu\right)=\frac{rq_{r}}{1-q_{r}}.
\]
\end{thm}

We would like to point out that all examples discussed in \cite{KN22b}
for which the quantization dimension exists (\emph{self-similar measure,
inhomogeneous self-similar measure, Gibbs measures with possible overlap})
remain literally valid for $r\in\left(-\dim_{\infty}\left(\nu\right),0\right)$
as well.
\begin{example}
\label{exa:Menger_Sponge} We briefly discuss one particularly regular
example. We consider the self-similar measure $\nu$ supported on
a\emph{ dyadic Menger sponge} in $\R^{3}$ with the four defining
contractions 
\[
\left\{ \phi_{k}:z\mapsto1/2\cdot z+k:k\in\left\{ \left(0,0,0\right),\left(1/2,0,0\right),\left(0,1/2,0\right),\left(0,0,1/2\right)\right\} \right\} 
\]
\begin{figure}[h]
\center{
\begin{tikzpicture}[scale=.8, every node/.style={transform shape},line cap=round,line join=round,>=triangle 45,x=1cm,y=1cm] 
\begin{axis}[ x=2.7cm,y=2.7cm, axis lines=middle, axis line style={very thick},ymajorgrids=false, xmajorgrids=false, grid style={thick,densely dotted,black!20}, xlabel= {$q$}, ylabel= {$\beta_\nu (q)$}, xmin=-0.49 , xmax=2.4 , ymin=-1.2, ymax=2.5,x tick style={thick, color=black}, xtick={0, 1,1.87},xticklabels = {0,1,$q_r$}, y tick style={thick, color=black}, ytick={-1/2,0,-log2(0.66),0.92/(1.86-1),1.35,2},yticklabels = {$r$,0,$\overline{\dim}_\infty(\nu)$,$ {D}_{r}(\nu)$,$ {D}_{0}(\nu)$,2}] 
\draw [line width=01pt,dashdotted,    domain=-0.05 :1.86] plot(\x,{1.35*(1-\x)});
\draw [line width=01pt,dashed,    domain=-0.05 :2.3] plot(\x,{-log2(0.66)*(1-\x)});
\draw[line width=1pt,smooth,samples=180,domain=-0.1:2.18] plot(\x,{log10(0.08^((\x))+0.2^((\x))+0.06^((\x))+0.66^((\x)))/log10(2)}); 
\draw [line width=01pt,solid, domain=-0.05 :2.1] plot(\x,{0.92/(1.86-1)*(1-\x)});
\draw [line width=01pt,loosely dotted , domain=-0.05 :2.3] plot(\x,{-.5*\x});


\draw [line width=.7pt,dotted, gray] (1.87 ,-0.2)--(1.87,-1); 
\draw [line width=.7pt,dotted, gray] (1 ,-0.2)--(1,-1/2);
\draw [line width=.7pt,dotted, gray] (0,-0.5)--(1,-.5);
\draw (0.07 ,2.13 ) node[anchor=north west] {$ \displaystyle{\overline{\dim}_M(\nu)}$}; 
\end{axis} 
\end{tikzpicture}}

\caption{\label{fig:Moment-generating-function} For the $L^{q}$-spectrum
of the self-similar measure $\nu$ supported on a\emph{ dyadic Menger
sponge} in $\R^{3}$ with four contractions and with probability vector
$\left(0.66,0.2,0.08,0.06\right)$ we have $\beta_{\nu}\left(q\right)=\tau_{\J_{\nu,r}}\left(q\right)+qr$,
$\beta_{\nu}\left(0\right)=2$ and $\dim_{\infty}\left(\nu\right)=\log0.66/\log2<-0.599$.
For $r=-0.5>-\dim_{\infty}\left(\nu\right)$ the intersection of the
graph of $\beta_{\nu}$ and the dashed line determines $q_{r}$. The
(solid) line through the points $\left(q_{r},\beta_{\nu}\left(q_{r}\right)\right)$
and $\left(1,0\right)$ intersects the vertical axis in $D_{r}\left(\nu\right)$.
The (dash-dotted) tangent to $\beta_{\nu}$ in $1$ intersects the
vertical axis in $D_{0}\left(\nu\right)$. Also the lower bound $D_{r}\left(\nu\right)\protect\geq\dim_{\infty}\left(\nu\right)$
becomes obvious.}
\end{figure}
and probability vector $\left(0.66,0.2,0.08,0.06\right)$. In this
example, we have that $\tau_{\J_{\nu,r}}\left(q\right)=\beta_{\nu}\left(q\right)-qr$
even for $r<0$ and the $L^{q}$-spectrum exists as a limit and is
differentiable on $\R_{>0}$. Furthermore, we have $\beta_{\nu}\left(0\right)=\overline{\dim}_{M}\left(\nu\right)=2$.
Therefore, our second main result (see \prettyref{thm:LqRegularImpliesRegular})
on regularity applies and we can read off the value of $q_{r}$ and,
as a consequence of \prettyref{thm:Main} also of $D_{r}\left(\nu\right)$,
directly from $\beta_{\nu}$ for all values $r>-\dim_{\infty}\left(\nu\right)=\log_{2}0.66\thickapprox-0.599$
as demonstrated in \prettyref{fig:Moment-generating-function}.
\end{example}

\section{\label{sec:Upper-bounds-1} Proof of main theorems}

In this section, we only give a proof for the upper and lower bounds
of the quantization dimension of order $r\in\left(-\dim_{\infty}\left(\nu\right),0\right)$
as stated in \prettyref{thm:Main} \prettyref{enu:2} and the proof
of part \prettyref{enu:3}. This is sufficient for proving \prettyref{thm:Main},
since part \prettyref{enu:1} is fully covered by \cite{KN22b}. We
conclude this section by outlining the proof of the theorem on absolutely
continuous measures, \prettyref{thm:absoltutelyContinuousCase}.

\subsection{Optimal partitions and partition entropy\label{sec:OptimalPartitions}}

We make use of some general observations from \cite{MR4882804} which
is valid for arbitrary set functions $\J:\mathcal{D}\to\R_{\geq0}$
on the dyadic cubes $\mathcal{D}$, which are monotone, $\dim_{\infty}(\J)>0$
(in particular uniformly vanishing) and locally non-vanishing with
$\J\left(\Q\right)>0$ and such that $\liminf_{n}\GL_{\J,n}\left(q\right)\in\R$
for some $q>0$. Here, \emph{uniformly vanishing} means $\lim_{k\to\text{\ensuremath{\infty}}}\sup_{Q\in\bigcup_{n\geq k}\mathcal{D}_{n}}\J\left(Q\right)=0$
and \emph{locally non-vanishing} means $\J\left(Q\right)>0$ implies
that there exists $Q'\in\mathcal{D}\left(Q\right)\setminus\left\{ Q\right\} $
with $\J\left(Q'\right)>0$. It is important to note that all these
conditions on the set function are fulfilled for our particular choice
$\J=\J_{\nu,r}$ whenever $r>-\dim_{\infty}\left(\nu\right)$ since
$\dim_{\infty}(\J_{\nu,r})=\dim_{\infty}(\nu)+r$.

We also recall the closely connected \emph{dual problem,} where we
consider 
\[
\gamma_{\J,n}\coloneqq\inf_{P\in\Pi,\card(P)\leq n}\max_{Q\in P}\J\left(Q\right)
\]
and convergence rates
\[
\overline{\alpha}_{\J}\coloneqq\limsup_{n\rightarrow\infty}\frac{\log\left(\gamma_{\J,n}\right)}{\log(n)}\;\;\text{and}\;\;\underline{\alpha}_{\J}\coloneqq\liminf_{n\rightarrow\infty}\frac{\log\left(\gamma_{\J,n}\right)}{\log(n)}.
\]
With this at hand, we can state the crucial results used in the proofs
of our main theorem as follows.
\begin{prop}[{\cite[Theorems 1.4, 1.8 \& Section 1.3]{MR4882804}}]
\label{prop:GeneralResultOnPartitionFunction} We have
\[
\overline{F}_{\J}=\overline{h}_{\J}=\frac{-1}{\overline{\alpha}_{\J}}=q_{\J}\quad\text{and }\quad\underline{F}_{\J}\leq\underline{h}_{\J}=\frac{-1}{\underline{\alpha}_{\J}}.
\]
\end{prop}

For the proof of the main theorem, we divide the problem into upper
and lower bounds.

\subsection{Upper bounds}

We first prove the upper bounds, which are somewhat less demanding
than the lower bounds discussed thereafter. In this section, the underlying
norm on $\R^{d}$ is assumed to be euclidean.
\begin{prop}
For $r\in\left(-\dim_{\infty}\left(\nu\right),0\right)$, we have
\[
\underline{D}_{r}\left(\nu\right)\leq\frac{r\overline{F}_{r}}{1-\overline{F}_{r}}=\frac{r\overline{h}_{r}}{1-\overline{h}_{r}}=\frac{rq_{r}}{1-q_{r}}\;
\]
and if $\underline{F}_{r}>1$,
\[
\overline{D}_{r}\left(\nu\right)\leq\frac{r\underline{F}_{r}}{1-\underline{F}_{r}}.
\]
\end{prop}

\begin{proof}
The second and third equality follows from \prettyref{prop:GeneralResultOnPartitionFunction}.
For fixed $\alpha>0$ define $c_{\alpha,n}\coloneqq\card\left(\mathcal{N}_{\nu,\alpha,r}\left(n\right)\right)$.
We begin with the upper bound for the lower quantization dimension.
For each $Q\in\mathcal{N}_{\nu,\alpha,r}\left(n\right)$ we consider
$Q'\in\mathcal{D}\left(Q\right)$ such that 
\[
\nu\left(Q'\right)\Lambda\left(Q'\right)^{r/d}=\max_{C\in\mathcal{D}\left(Q\right)}\nu\left(C\right)\Lambda\left(C\right)^{r/d}=\J_{\nu,r}\left(Q\right)
\]
 and let $A$ denote the set of all centers of the elements $Q'$
for $Q\in\mathcal{N}_{\nu,\alpha,r}\left(n\right)$. Then,
\[
\sup_{x\in Q'}d(x,A)\leq\sqrt{d}\Lambda\left(Q'\right)^{1/d}
\]
 keeping in mind that $r$ is negative,
\begin{align*}
\er_{c_{\alpha,n},r}\left(\nu\right) & \leq\left(\int d(x,A)^{r}\d\nu(x)\right)^{1/r}\leq\left(\sum_{Q\in\mathcal{N}_{\nu,\alpha,r}\left(n\right)}\int_{Q'}d(x,A)^{r}\d\nu(x)\right)^{1/r}\\
 & \leq\sqrt{d}\left(\sum_{Q\in\mathcal{N}_{\nu,\alpha,r}\left(n\right)}\Lambda\left(Q'\right)^{r/d}\nu\left(Q'\right)\right)^{1/r}=\sqrt{d}\left(\sum_{Q\in\mathcal{N}_{\nu,\alpha,r}\left(n\right)}\J_{\nu,r}\left(Q\right)\right)^{1/r}\\
 & \leq\sqrt{d}c_{\alpha,n}^{1/r}2^{-\alpha n/r}.
\end{align*}
Since $\overline{F}_{r}=q_{r}>1$ we consider only $\alpha$ in $\overline{\mathscr{A}}\coloneqq\left\{ \alpha>0:\overline{F}_{r}\left(\alpha\right)/\alpha>1\right\} \neq\emptyset$.
For such $\alpha\in\overline{\mathscr{A}}$, take a subsequence $\left(n_{k}\right)$
such that $\lim_{k}\log c_{\alpha,n_{k}}/n_{k}=\overline{F}_{r}\left(\alpha\right)>\alpha>0$
and $c_{\alpha,n_{k}}^{1/r}2^{-\alpha n_{k}/r}<1$. Then 
\[
\frac{\log c_{\alpha,n_{k}}}{-\log\er_{c_{\alpha,n_{k}},r}\left(\nu\right)}\leq\frac{r\log c_{\alpha,n_{k}}}{-r/2\log d-\log c_{\alpha,n_{k}}+\alpha n_{k}}\leq\frac{r\log c_{\alpha,n_{k}}/\left(\alpha n_{k}\right)}{1-r\log\left(d\right)/\left(2\alpha n_{k}\right)-\log c_{\alpha,n_{k}}/\left(\alpha n_{k}\right)}
\]
and therefore
\begin{align*}
\underline{D}_{r}\left(\nu\right)=\liminf_{n\to\infty}\frac{\log n}{-\log\er_{n,r}\left(\nu\right)} & \leq\lim_{k\to\infty}\frac{\log c_{\alpha,n_{k}}}{-\log\er_{c_{\alpha,n_{k}},r}\left(\nu\right)}\leq\frac{r\overline{F}_{r}\left(\alpha\right)/\alpha}{1-\overline{F}_{r}\left(\alpha\right)/\alpha}.
\end{align*}
Taking the infimum over $\alpha\in\overline{\mathscr{A}}$ and, keeping
in mind that $r<0$, yields
\[
\underline{D}_{r}\left(\nu\right)\leq\inf_{\alpha\in\overline{\mathscr{A}}}\frac{r\overline{F}_{r}\left(\alpha\right)/\alpha}{1-\overline{F}_{r}\left(\alpha\right)/\alpha}=\frac{r\sup_{\alpha>0}\overline{F}_{r}\left(\alpha\right)/\alpha}{1-\sup_{\alpha>0}\overline{F}_{r}\left(\alpha\right)/\alpha}=\frac{r\overline{F}_{r}}{1-\overline{F}_{r}}=\frac{rq_{r}}{1-q_{r}}.
\]

Finally, we show the upper bound for the upper quantization dimension
under the assumption $\underline{F}_{r}>1$. Let $Q\in E_{\alpha,n}$,
$Q'\in\mathcal{D}\left(Q\right)$ and $A$ be given as above. Then
for $\alpha>0$ such that $\underline{F}_{r}\left(\alpha\right)>0$
and every $\epsilon\in$$\left(0,\underline{F}_{r}\left(\alpha\right)\right)$
and $n$ large, we have $c_{\alpha,n}\geq2^{n\left(\underline{F}_{r}\left(\alpha\right)-\epsilon\right)}$.
With $n_{k}\coloneqq\left\lfloor \log_{2}\left(k\right)/\left(\underline{F}_{r}\left(\alpha\right)-\epsilon\right)\right\rfloor $,
$k\in\N$, we find
\begin{align*}
\er_{k,r}\left(\nu\right) & \leq\er_{c_{\alpha,n_{k}},r}\left(\nu\right)\leq\sqrt{d}c_{\alpha,n_{k}}^{1/r}2^{-\alpha n_{k}/r}.
\end{align*}
Since $\underline{F}_{r}>1$ we have $\alpha\in\mathscr{\underline{A}}\coloneqq\left\{ \alpha>0:\underline{F}_{r}\left(\alpha\right)/\alpha>1\right\} \neq\emptyset$
and $\sqrt{d}c_{\alpha,n_{k}}^{1/r}2^{-\alpha n_{k}/r}<1$ for $k$
large. Hence, we find
\begin{align*}
\frac{\log k}{-\log\er_{k,r}\left(\nu\right)} & \leq\frac{r\log k}{-r/2\log d-\log c_{\alpha,n_{k}}+\alpha n_{k}\log\left(2\right)}\\
 & \leq\frac{r\log k}{-r/2\log d-\log k+\alpha\log\left(k\right)/\left(\underline{F}_{r}\left(\alpha\right)-\epsilon\right)}.
\end{align*}
Taking the upper limit and since $\underline{F}_{r}\left(\alpha\right)/\alpha>1$
we find 
\begin{align*}
\overline{D}_{r}\left(\nu\right) & \leq\frac{r\underline{F}_{r}\left(\alpha\right)/\alpha}{1-\underline{F}_{r}\left(\alpha\right)/\alpha}.
\end{align*}
Taking the infimum over $\alpha\in\underline{\mathscr{A}}$ yields
\begin{align*}
\overline{D}_{r}\left(\nu\right) & \leq\inf_{\alpha\in\underline{\mathscr{A}}}\frac{r\underline{F}_{r}\left(\alpha\right)/\alpha}{1-\underline{F}_{r}\left(\alpha\right)/\alpha}=\frac{r\sup_{\alpha>0}\underline{F}_{r}\left(\alpha\right)/\alpha}{1-\sup_{\alpha>0}\underline{F}_{r}\left(\alpha\right)/\alpha}=\frac{r\underline{F}_{r}}{1-\underline{F}_{r}}.
\end{align*}
\end{proof}

\subsection{Lower bounds \label{subsec:Lower-bounds}}

In this section we give the proof of the following lower bounds.
\begin{prop}
\label{prop:LowerBound} If $r\in\left(-\dim_{\infty}\left(\nu\right),0\right)$,
then $q_{r}>1$ and we have 
\[
\frac{rq_{r}}{1-q_{r}}\leq\underline{D}_{r}\left(\nu\right)\;\text{and }\;\lim_{\varepsilon\downarrow0}\frac{r\underline{h}_{r-\varepsilon}}{1-\underline{h}_{r-\varepsilon}}\leq\overline{D}_{r}\left(\nu\right).
\]
\end{prop}

For convenient, in this section we choose the maximum norm on $\R^{d}$.
For any $Q\in\mathcal{D}$, we let $\left|Q\right|$ denote the side
length of $Q$. Before we proceed, we need the following two elementary
lemmas. For $\tilde{Q}\in\mathcal{D}$ we let $B_{n}\left(\tilde{Q}\right)\coloneqq\bigcup\left\{ Q\in\mathcal{D}_{n-1}:\overline{\tilde{Q}}\cap\overline{Q}\neq\emptyset\right\} $
denote the $2^{-n+1}$-parallel set of $\tilde{Q}$.
\begin{lem}
\label{lem:Combinatoric1}For $\tilde{Q}\in\mathcal{D}$, a finite
set $A\subset\R^{d}$, and an integer $n>\left|\log_{2}\left|\tilde{Q}\right|\right|$,
we have
\begin{align*}
\card\left\{ Q\in\mathcal{D}_{n}\left(\tilde{Q}\right):d\left(Q,A\right)\geq\left|Q\right|\&\forall Q'\in\mathcal{D}_{n-1}\left(\tilde{Q}\right):Q'\supset Q,d\left(Q',A\right)<\left|Q'\right|\right\} \\
\leq6^{d}\card\left(A\cap B_{n}\left(\tilde{Q}\right)\right).
\end{align*}
\end{lem}

\begin{proof}
First we show that for any $a\in\R^{d}$ we have 
\begin{align*}
\card\left\{ Q\in\mathcal{D}_{n}\left(\tilde{Q}\right):d\left(Q,a\right)\geq\left|Q\right|\&\forall Q'\in\mathcal{D}_{n-1}\left(\tilde{Q}\right):Q'\supset Q,d\left(Q',a\right)<\left|Q'\right|\right\} \\
\leq6^{d}\1_{B_{n}\left(\tilde{Q}\right)}\left(a\right).
\end{align*}
The case $a\notin B_{n}\left(\tilde{Q}\right)$ follows by observing
that $n>\left|\log_{2}\left|\tilde{Q}\right|\right|$ implies that
for every $Q\in\mathcal{D}_{n}\left(\tilde{Q}\right)$, there exists
$Q'\in\mathcal{D}_{n-1}\left(\tilde{Q}\right)$ with $Q'\supset Q$
and $d\left(Q',a\right)\geq d\left(\tilde{Q},a\right)\geq2^{-n+1}=\left|Q'\right|$
and the relevant set is therefore empty.

For the case $a\in B_{n}\left(\tilde{Q}\right)$, observe that we
have at most $3^{d}$ cubes $Q'\in\mathcal{D}_{n-1}\left(\tilde{Q}\right)$
such that $d\left(Q',a\right)<\left|Q'\right|$ and in each such cube
there are at most $2^{d}$ subcubes from $\mathcal{D}_{n}$. This
gives the upper bound for the cardinality in question of $2^{d}\cdot3^{d}$.

Now the claim follows from 
\begin{align*}
 & \card\left\{ Q\in\mathcal{D}_{n}\left(\tilde{Q}\right):d\left(Q,A\right)\geq\left|Q\right|\&\forall Q'\in\mathcal{D}_{n-1}\left(\tilde{Q}\right):Q'\supset Q,d\left(Q',A\right)<\left|Q'\right|\right\} \\
 & \leq\sum_{a\in A}\card\left\{ Q\in\mathcal{D}_{n}\left(\tilde{Q}\right):d\left(Q,a\right)\geq\left|Q\right|\&\forall Q'\in\mathcal{D}_{n-1}\left(\tilde{Q}\right):Q'\supset Q,d\left(Q',a\right)<\left|Q'\right|\right\} .
\end{align*}
\end{proof}
\begin{lem}
\label{lem:Combinatoric2} For $A\subset\R^{d}$ and $P\subset\bigcup_{k=1}^{n-1}\mathcal{D}_{k}$,
$n\in\N$ a finite disjoint family of sets we have
\[
\sum_{\tilde{Q}\in P}\card\left(A\cap B_{n}\left(\tilde{Q}\right)\right)\leq3^{d}\card\left(A\right).
\]
\end{lem}

\begin{proof}
Since for all $\tilde{Q}\in P$ we have $\left|\tilde{Q}\right|\geq2^{-n+1}$
we find
\begin{align*}
\sum_{\tilde{Q}\in P}\card\left(A\cap B_{n}\left(\tilde{Q}\right)\right) & =\sum_{\tilde{Q}\in P}\sum_{Q\in\D_{n-1},\bar{Q}\cap\bar{\tilde{Q}}\neq\emptyset}\card\left(A\cap Q\right)\\
 & =\sum_{Q\in\D_{n-1}}\sum_{\tilde{Q}\in P,\bar{Q}\cap\bar{\tilde{Q}}\neq\emptyset}\card\left(A\cap Q\right)\\
 & \leq\sum_{Q\in\D_{n-1}}\sum_{\tilde{Q}\in\mathcal{D}_{n-1},\bar{Q}\cap\bar{\tilde{Q}}\neq\emptyset}\card\left(A\cap Q\right)\\
 & \leq\sum_{Q\in\D_{n-1}}3^{d}\card\left(A\cap Q\right)=3^{d}\card\left(A\right).
\end{align*}
\end{proof}
\begin{proof}[Proof of \prettyref{prop:LowerBound}]
 For $k\in\N$ choose $A\in\mathcal{A}_{k}$, and $\epsilon>0$ such
that $r-\epsilon>-\dim_{\infty}\left(\nu\right)$. Let $t_{k}\coloneqq\gamma_{\J_{\nu,r-\epsilon},k}$
be as in \prettyref{prop:GeneralResultOnPartitionFunction} and $P\in\Pi$
an optimal partition realizing $\gamma_{\J_{\nu,r-\epsilon},k}$ that
is $\card\left(P\right)\leq k$ and $\J_{\nu,r-\epsilon}$$\left(\tilde{Q}\right)\leq t_{k}$
for all $\tilde{Q}\in P$ (see \cite{MR4882804}). Let us define 
\[
P_{1}\coloneqq\left\{ \tilde{Q}\in P,d\left(\tilde{Q},A\right)\geq\left|\tilde{Q}\right|\right\} \,\text{and }P_{2}\coloneqq P\setminus P_{1}.
\]
This allows us to estimate

\begin{align*}
\int d(x,A)^{r}\d\nu(x) & =\sum_{\tilde{Q}\in P_{1}}\int_{\tilde{Q}}d(x,A)^{r}\d\nu(x)+\sum_{\tilde{Q}\in P_{2}}\int_{\tilde{Q}}d(x,A)^{r}\d\nu(x)\\
 & \leq\sum_{\tilde{Q}\in P_{1}}\underbrace{\Lambda\left(\tilde{Q}\right)^{r/d}\nu\left(\tilde{Q}\right)}_{\leq t_{k}}+\sum_{\tilde{Q}\in P_{2}}\int_{\tilde{Q}}d(x,A)^{r}\d\nu(x)\\
 & \leq t_{k}\card(P)+\sum_{\tilde{Q}\in P_{2}}\int_{\tilde{Q}}d(x,A)^{r}\d\nu(x).
\end{align*}
Further, for $n\in\N$, and $\tilde{Q}\in P_{2}$ we set 
\begin{align*}
E_{n}\left(\tilde{Q}\right) & \coloneqq\left\{ Q\in\mathcal{D}_{n}\left(\tilde{Q}\right):d\left(Q,A\right)\geq\left|Q\right|\&\forall Q'\in\mathcal{D}_{n-1}\left(\tilde{Q}\right):Q'\supset Q,d\left(Q',A\right)<\left|Q'\right|\right\} .
\end{align*}
Note that for $2^{-n}>\left|\tilde{Q}\right|$ we have $\mathcal{D}_{n}\left(\tilde{Q}\right)=\emptyset$
and therefore $E_{n}\left(\tilde{Q}\right)=\emptyset$. For $2^{-n}<\left|\tilde{Q}\right|$
and $Q\in\mathcal{D}_{n}\left(\tilde{Q}\right)$ there is exactly
one $Q'\in\mathcal{D}_{n-1}\left(\tilde{Q}\right)$ with $Q'\supset Q$.
Finally, if $d\left(\tilde{Q},A\right)<\left|\tilde{Q}\right|=2^{-n}$
we again have $E_{n}\left(\tilde{Q}\right)=\emptyset$. Since $\nu$
has no atoms,
\[
\bigcup_{n\in\N}\bigcup_{\tilde{Q}\in P_{2}}E_{n}\left(\tilde{Q}\right)=\bigcup_{\tilde{Q}\in P_{2}}\tilde{Q}\setminus A,
\]
i.\,e\@. $\left(E_{n}\left(\tilde{Q}\right)\right)_{n\in\N,\tilde{Q}\in P_{2}}$
is a countable $\nu$-almost sure infinite partition of \emph{$\bigcup_{\tilde{Q}\in P_{2}}\tilde{Q}$}.
With this at hand, we find for the second summand in the above estimate
\begin{align}
\sum_{\tilde{Q}\in P_{2}}\int d(x,A)^{r}\d\nu(x) & \leq\sum_{n=0}^{\infty}\sum_{\tilde{Q}\in P_{2}}\sum_{Q\in E_{n}\left(\tilde{Q}\right)}\int_{Q}d\left(x,A\right)^{r}\d\nu(x)\nonumber \\
 & \leq\sum_{n=0}^{\infty}\sum_{\tilde{Q}\in P_{2}}\sum_{Q\in E_{n}\left(\tilde{Q}\right)}\underbrace{\Lambda\left(Q\right)^{\left(r/d-\epsilon/d\right)}\nu\left(Q\right)}_{\leq t_{k}}\underbrace{\Lambda\left(Q\right)^{\epsilon/d}}_{=2^{-n\epsilon}}\nonumber \\
 & \leq t_{k}\sum_{n=0}^{\infty}2^{-\epsilon n}\sum_{\tilde{Q}\in P_{2},\left|\tilde{Q}\right|>2^{-n}}6^{d}\card\left(A\cap B_{n}\left(\tilde{Q}\right)\right)\nonumber \\
 & \leq t_{k}\sum_{n=0}^{\infty}2^{-\epsilon n}6^{d}\sum_{\tilde{Q}\in P_{2},\left|\tilde{Q}\right|>2^{-n}}\card\left(A\cap B_{n}\left(\tilde{Q}\right)\right)\nonumber \\
 & \leq\frac{18^{d}}{1-2^{-\epsilon}}t_{k}\card\left(A\right),\label{eq:upperbound_e_r}
\end{align}
where for the third inequality we used \prettyref{lem:Combinatoric1}
and for the fifth inequality we used \prettyref{lem:Combinatoric2}.
Combining the above then gives
\[
\int d(x,A)^{r}\d\nu(x)\leq\left(1+\frac{18^{d}}{1-2^{-\epsilon}}\right)t_{k}k.
\]
With \prettyref{eq:upperbound_e_r} this gives
\[
\er_{r,k}\left(\nu\right)\geq\left(\left(1+\frac{18^{d}}{1-2^{-\epsilon}}\right)t_{k}k\right)^{1/r}.
\]
Now, using $q_{r-\varepsilon}+\epsilon\geq\log\left(n\right)/\left(-\log\left(t_{n}\right)\right)$
for all $n$ large and \prettyref{prop:GeneralResultOnPartitionFunction},
gives
\[
\underline{D}_{r}\left(\nu\right)\geq\liminf_{k\to\infty}\frac{r\log k}{-\log k-\log t_{k}}=\frac{r\limsup_{k\to\infty}\log\left(k\right)/\log\left(1/t_{k}\right)}{1-\limsup_{k\to\infty}\log\left(k\right)/\log\left(1/t_{k}\right)}\geq\frac{r\left(q_{r-\varepsilon}+\epsilon\right)}{1-\left(q_{r-\varepsilon}+\epsilon\right)}.
\]
Letting $\epsilon$ tend to zero and by the continuity of $a\mapsto q_{a}$,
$a\in(-\dim_{\infty}\left(\nu\right),0)$ the first inequality follows.

For the second inequality, recall that by our assumption and \prettyref{prop:GeneralResultOnPartitionFunction}
we have $1<\underline{F}_{\J}\leq\underline{h}_{\J}$. For the second
inequality we consider $P\in\Pi$ with $\card\left(P\right)=\mathcal{M}_{\nu,r-\epsilon}\left(1/t\right)$.
This time \prettyref{eq:upperbound_e_r} gives
\[
\er_{r,\mathcal{M}_{\nu,r-\epsilon}\left(1/t\right)}\left(\nu\right)\geq\left(\frac{18^{d}}{1-2^{-\epsilon}}t\mathcal{M}_{\nu,r-\epsilon}\left(1/t\right)\right)^{1/r}
\]
and hence using $\underline{h}_{r-\epsilon}+\epsilon\geq\log\left(\mathcal{M}_{\nu,r-\epsilon}\left(1/t_{n}\right)\right)/\log\left(1/t_{n}\right)$
on a subsequence $t_{n}\searrow0$,
\begin{align*}
\overline{D}_{r}\left(\nu\right) & \geq\limsup_{t\to0}\frac{r\log\mathcal{M}_{\nu,r-\epsilon}\left(1/t\right)}{-\log\mathcal{M}_{\nu,r-\epsilon}\left(1/t\right)+\log\left(1/t\right)}=\frac{r\liminf\log\left(n\right)/\log\left(1/t_{n}\right)}{1-\liminf\log\left(n\right)/\log\left(1/t_{n}\right)}\\
 & \geq\frac{r\left(\underline{h}_{r-\epsilon}+\epsilon\right)}{1-\left(\underline{h}_{r-\epsilon}+\epsilon\right)}.
\end{align*}
\end{proof}
\begin{proof}[Proof of \prettyref{lem:GeometricBounds}]
 First note that $a_{\nu}\coloneqq\sup\left\{ q_{r}:r\in\left(-\dim_{\infty}\left(\nu\right),0\right)\right\} >1,$
due to the fact that $q_{0}=1$ and $r\mapsto q_{r}$ is strictly
decreasing on $\left(-\dim_{\infty}\left(\nu\right),0\right]$. Since
$r\mapsto q_{r}$ is decreasing we have 
\[
\lim_{r\searrow-\dim_{\infty}\left(\nu\right)}q_{r}=\sup\left\{ q_{r}:r\in\left(-\dim_{\infty}\left(\nu\right),0\right)\right\} =a_{\nu},
\]
and therefore 
\[
\lim_{r\searrow-\dim_{\infty}\left(\nu\right)}\underline{D}_{r}\left(\nu\right)=\begin{cases}
{\displaystyle \lim_{r\searrow-\dim_{\infty}\left(\nu\right)}r\frac{q_{r}}{1-q_{r}}=\frac{a_{\nu}}{a_{\nu}-1}\dim_{\infty}\left(\nu\right)}\vphantom{\frac{df\frac{G}{sdG}}{\frac{dfG}{dfG\frac{d}{d}}}}, & a_{\nu}<\infty,\\
{\displaystyle \lim_{r\searrow-\dim_{\infty}\left(\nu\right)}\frac{r}{1/q_{r}-1}=\dim_{\infty}\left(\nu\right),} & a_{\nu}=\infty.
\end{cases}
\]
\end{proof}

\subsection{Quantization via geometric mean error}

In this subsection we settle the quantziation problem for the geometric
mean error as stateted in \prettyref{thm:Main} \prettyref{enu:3}.
\begin{proof}[Proof of \prettyref{thm:Main} \prettyref{enu:3}]
 Part \prettyref{enu:3} follows by combining \prettyref{lem:monontonicity}
with \cite{MR2954531}, \cite{MR2334791} and our results from part
\prettyref{enu:2}, by noting, on the one hand,
\begin{equation}
\frac{-1}{\partial_{r}q_{r}|_{r=0}}=\lim_{r\uparrow0}\frac{r}{1-q_{r}}q_{r}\leq\lim_{r\uparrow0}\frac{rq_{r}}{1-q_{r}}=\lim_{r\uparrow0}\underline{D}_{r}\left(\nu\right)\leq\underline{D}_{0}\left(\nu\right),\label{eq:-1/q_r'<D_o}
\end{equation}
where we used $\lim_{r\uparrow0}q_{r}=1$. Using $\beta_{\nu}\left(q_{r}\right)\leq\tau_{\nu,r}\left(q_{r}\right)+rq_{r}=rq_{r}$,
this also shows 
\begin{equation}
\frac{-1}{\partial_{r}q_{r}|_{r=0}}\leq\lim_{r\uparrow0}\frac{\beta_{\nu}\left(q_{r}\right)}{1-q_{r}}=-\beta'_{\nu}\left(1+\right).\label{eq:-1/q_r'<beta'}
\end{equation}
On the other hand, using $\beta_{\nu}\left(q_{r}\right)=rq_{r}$ for
$r>0$, as established in \cite{KN22b} and the monotonicity of $r\mapsto\overline{D}_{r}\left(\nu\right)$
on $r\geq0$ obtained e.\,g\@. in \cite[Lemma 3.5]{MR2055056},
we have that 
\begin{align}
\overline{D}_{0}\left(\nu\right) & \leq\lim_{r\downarrow0}\overline{D}_{r}\left(\nu\right)=\lim_{r\downarrow0}\frac{\beta_{\nu}(q_{r})-\beta_{\nu}(1)}{1-q_{r}}=-\beta_{\nu}'\left(1-\right)\nonumber \\
 & =\lim_{r\downarrow0}\frac{r}{1-q_{r}}q_{r}=\frac{-1}{\partial_{r}q_{r}|_{r=0}}.\label{eq:D_0>-1/q_r'}
\end{align}
Combining $-\beta_{\nu}'\left(1+\right)\leq-\beta_{\nu}'\left(1-\right)$,
\ref{eq:-1/q_r'<beta'}, and \ref{eq:D_0>-1/q_r'} shows that $\beta_{\nu}$
is differentiable at $1$ with $\beta_{\nu}'\left(1\right)=1/\partial_{r}q_{r}|_{r=0}$.
Combining \ref{eq:-1/q_r'<D_o} and \ref{eq:D_0>-1/q_r'} $-1/\partial_{r}q_{r}|_{r=0}\leq\underline{D}_{0}\left(\nu\right)\leq\overline{D}_{0}\left(\nu\right)\leq-1/\partial_{r}q_{r}|_{r=0}$.
This proves the claim.
\end{proof}

\subsection{Absolutely continuous case\label{subsec:Absolutely-continuous-case}}

We start with the following lemma, which has been stated in a similar
context in \cite[Lemma 3.15]{KN2022b}.
\begin{lem}
\label{lem:AbsolLq}Let $\nu$ be a non-zero absolutely continuous
measure with Lebesgue density $f\in L_{\Lambda}^{s}$ for some $s\geq1$.
Then, for all $q\in\left[0,s\right]$, the $L^{q}$-spectrum is linear
with $\liminf_{n\rightarrow\infty}\beta_{\nu,n}(q)=\beta_{\nu}\left(q\right)=d(1-q)$
and $\tau_{\J_{\nu,r}}(q)=\beta_{\nu}(q)-rq$ for $r\geq d/s-d$.
\end{lem}

\begin{proof}
First, we remark that, $\beta_{\nu}(1)=0$ and $\beta_{\nu}(0)=d$.
Hence, the convexity of $\beta_{\nu}$ implies $\beta_{\nu}(q)\leq d(1-q)$
for all $q\in[0,1]$ and $\beta_{\nu}(q)\geq d(1-q)$ for $q>1$.
Moreover, by Jensen's inequality, for all $q\in[0,1]$ and $n$ large,
we have 
\[
\sum_{Q\in\mathcal{D}_{n}}\nu(Q)^{q}=\sum_{Q\in\mathcal{D}_{n}}\left(\frac{\int_{Q}f\d\Lambda}{\Lambda(Q)}\right)^{q}\Lambda(Q)^{q}\geq\sum_{Q\in\mathcal{D}_{n}}\Lambda(Q)^{q-1}\int_{Q}f^{q}\d\Lambda\geq\Lambda(Q)^{q-1}\int_{\Q}f^{q}\d\Lambda,
\]
 implying $\liminf_{n\rightarrow\infty}\beta_{\nu,n}(q)\geq d(1-q).$
Further, Jensen's inequality, for all $q\in[1,s]$, yields 
\[
\sum_{Q\in\mathcal{D}_{n}}\nu(Q)^{q}=\sum_{Q\in\mathcal{D}_{n}}\left(\frac{\int_{Q}f\d\Lambda}{\Lambda(Q)}\right)^{q}\Lambda(Q)^{q}\leq\Lambda(Q)^{q-1}\sum_{Q\in\mathcal{D}_{n}}\int_{Q}f^{q}\d\Lambda\leq\Lambda(Q)^{q-1}\int_{\Q}f^{q}\d\Lambda.
\]
 Hence, we obtain $\limsup_{n\rightarrow\infty}\beta_{\nu,n}(q)\leq d(1-q).$
To prove the last equality we again use Jensen's inequality, for $Q\in\D_{n}$,
$r\geq d/s-d$ and $q=s$: 
\[
\nu(Q)^{q}=\left(\int_{Q}f\Lambda(Q)^{-1}\d\Lambda\right)^{q}\Lambda(Q)^{q}\leq\left(\int_{Q}f^{q}\d\Lambda\right)\Lambda(Q)^{q-1},
\]
implying $\nu(Q)^{q}\Lambda(Q)^{qr/d}\leq\left(\int_{Q}f^{q}\d\Lambda\right)\Lambda(Q)^{q(1+r/d)-1}$.
Note that 
\[
q(1+r/d)-1=\left(\frac{d+r}{d}\right)q-1\geq\frac{q}{s}-1=0
\]
and since have that $Q\mapsto\left(\int_{Q}f^{q}\d\Lambda\right)\Lambda(Q)^{q(1+r/d)-1}$
is monotonic. Therefore, we get the following upper bound:

\begin{align*}
\sum_{Q\in\D_{n}}\max_{Q'\in\mathcal{D}\left(Q\right)}\nu\left(Q'\right)\Lambda\left(Q'\right)^{r} & \leq\sum_{Q\in\D_{n}}\left(\int_{Q}f^{q}\d\Lambda\right)\Lambda(Q)^{q(1+r/d)-1}\\
 & \leq\left\Vert f\right\Vert _{L_{\Lambda}^{q}(\Q)}^{q}2^{-q(d+r)+d}
\end{align*}
showing 
\[
\tau_{\J_{\nu,r}}(q)\leq d(1-q)-rq.
\]
Further, $\tau_{\J_{\nu,r}}(0)\leq d$ combined with the convexity
of $\tau_{\J_{\nu,r}}$, we conclude that for all $q\in[0,s]$
\[
\tau_{\J_{\nu,r}}(q)\leq d(1-q)-rq.
\]
The lower bound follows from 
\[
d(1-q)-rq\leq\liminf_{n\rightarrow\infty}\beta_{\nu,n}(q)-rq\leq\liminf_{n\rightarrow\infty}\tau_{\J_{\nu,r},n}(q)\leq\tau_{\J_{\nu,r}}(q).
\]
\end{proof}
We also need the following easy observation.
\begin{lem}
\label{lem:lowerboudonBeta} For any compactly supported probability
measure $\nu$ we have for $q\geq0$,
\[
\beta_{\nu}\left(q\right)\geq-\dim_{\infty}\left(\nu\right)\cdot q.
\]
\end{lem}

\begin{proof}
By the definition of $\dim_{\infty}\left(\nu\right)$ we have that
for every $\epsilon>0$ there exists infinitely many $n\in\N$ with
$\max_{Q\in\mathcal{D}_{n}}\nu\left(Q\right)\geq2^{\left(-\dim_{\infty}\left(\nu\right)-\epsilon\right)n}$
and therefore
\[
\limsup_{n\to\infty}\frac{1}{\log2^{n}}\log\sum_{Q\in\mathcal{D}_{n}}\nu(Q)^{q}\geq\left(-\dim_{\infty}\left(\nu\right)-\epsilon\right)q.
\]
\end{proof}
Combining the previous two lemmas gives $s_{h}\leq d/\left(d-\dim_{\infty}\left(h\Lambda\right)\right)$,
as claimed at the beginning of \prettyref{subsec:Absolutely-continuous-measures}.

For what follows, we take advantage of the fact that many basic inequalities
for positive $r$ can simply be reversed to hold for negative $r$.
To see this, recall from the introduction, that for negative $r$,
\[
V_{n,r}\left(\nu\right)=\er_{n,r}\left(\nu\right)^{r}=\begin{cases}
\inf_{A\in\mathcal{A}_{n}}\int d\left(x,A\right)^{r}\d\nu\left(x\right), & r>0,\\
\sup_{A\in\mathcal{A}_{n}}\int d\left(x,A\right)^{r}\d\nu\left(x\right), & r<0.
\end{cases}
\]

For the reversed inequalities of the following lemma for the case
$r>0$ we refer to \cite[Lemma 4.14]{MR1764176}.
\begin{lem}
\label{lem:Basic_convex}For $r<0$ and a linear combination $\nu\coloneqq\sum s_{i}\nu_{i}$,
$s_{i}\geq0$, of finite measures $\nu_{i}$ and $n\in\N$, we have
\[
V_{n,r}\left(\nu\right)\leq\sum s_{i}V_{n,r}\left(\nu_{i}\right).
\]
Further, for $n\geq\sum n_{i}$, $n_{i}\in\N$, we have 
\[
V_{n,r}\left(\nu\right)\geq\sum s_{i}V_{n_{i},r}\left(\nu_{i}\right).
\]
\end{lem}

\begin{proof}
For $A\in\mathcal{A}_{n}$, 
\[
\int d\left(x,A\right)^{r}\d\nu\left(x\right)=\sum s_{i}\int d\left(x,A\right)^{r}\d\nu_{i}\left(x\right)\leq\sum s_{i}V_{n,r}\left(\nu_{i}\right).
\]
Taking the supremum over $A\in\mathcal{A}_{n}$ gives the first inequality.

For the second inequality, assume $n\geq\sum n_{i}$ and note that
for $A_{i}\in\mathcal{A}_{n_{i}}$ we have $A\coloneqq\bigcup_{i}A_{i}\in\mathcal{A}_{n}$.
Hence,
\[
V_{n,r}\left(\nu\right)\geq\int d\left(x,A\right)^{r}\d\nu\left(x\right)=\sum s_{i}\int d\left(x,A\right)^{r}\d\nu_{i}\left(x\right)\geq\sum s_{i}\int d\left(x,A_{i}\right)^{r}\d\nu_{i}\left(x\right).
\]
Now taking the supremum over all $A_{i}\in\mathcal{A}_{n_{i}}$ gives
the desired second inequality.
\end{proof}
\begin{lem}
\label{lem:upper_bound_QE_Lebesgue} For $r\in\left(-d,0\right)$
we let $C_{r,d}\coloneqq2^{-r}+18^{d}/\left(2^{r}-2^{-d}\right)$.
Then for all $m\in\N$ and $A\in\mathcal{A}_{m}$ we have 
\[
\int d\left(x,A\right)^{r}\d\Lambda(x)\leq C_{r,d}m^{-r/d}.
\]
\end{lem}

\begin{proof}
Let us first consider the case $m=2^{nd}$ for $n\in\N$. We follow
the estimates in \prettyref{subsec:Lower-bounds} for the lower bound
of the quantization dimension with $\nu=\Lambda$. Observe that for
$n\in\N$ the optimal partition $P$ of cardinality $2^{nd}$ is in
this situation given by $\mathcal{D}_{n}$. As in the proof \prettyref{prop:LowerBound}
we partition $P$ into $P_{1}$, $P_{2}$ and obtain 
\begin{align*}
\int d(x,A)^{r}\d\Lambda(x) & =\sum_{\tilde{Q}\in P_{1}}\int_{\tilde{Q}}d(x,A)^{r}\d\Lambda(x)+\sum_{\tilde{Q}\in P_{2}}\int_{\tilde{Q}}d(x,A)^{r}\d\Lambda(x)\\
 & \leq\card(P)\cdot\underbrace{\Lambda\left(\tilde{Q}\right)^{r/d}\Lambda\left(\tilde{Q}\right)}_{=2^{-n(d+r)}}+\sum_{\tilde{Q}\in P_{2}}\int_{\tilde{Q}}d(x,A)^{r}\d\Lambda(x)\\
 & \leq m^{-r/d}+\sum_{\tilde{Q}\in P_{2}}\int_{\tilde{Q}}d(x,A)^{r}\d\Lambda(x).
\end{align*}
Using \prettyref{lem:Combinatoric1} and \prettyref{lem:Combinatoric2}
we estimate the second summand as follows: 
\begin{align*}
\sum_{\tilde{Q}\in P_{2}}\int d(x,A)^{r}\d\Lambda(x) & \leq\sum_{k=0}^{\infty}\sum_{\tilde{Q}\in P_{2}}\sum_{Q\in E_{k}\left(\tilde{Q}\right)}\int_{Q}d\left(x,A\right)^{r}\d\Lambda(x)\\
 & \leq\sum_{k=0}^{\infty}\sum_{\tilde{Q}\in P_{2}}\sum_{Q\in E_{k}\left(\tilde{Q}\right)}\underbrace{\Lambda\left(Q\right)^{r/d}\Lambda\left(Q\right)}_{=2^{-k(d+r)}}\\
 & \leq\sum_{k=n+1}^{\infty}\sum_{\tilde{Q}\in P_{2},\left|\tilde{Q}\right|>2^{-k}}6^{d}\card\left(A\cap B_{k}\left(\tilde{Q}\right)\right)\cdot2^{-k(d+r)}\\
 & \leq18^{d}\card(A)\sum_{k=n}^{\infty}2^{-k(d+r)}=\frac{18^{d}}{1-2^{-(d+r)}}m^{-r/d}.
\end{align*}
Combining the above, we obtain with $C:=\left(1+\frac{18^{d}}{1-2^{-(d+r)}}\right)$
\[
\int d(x,A)^{r}\d\Lambda(x)\leq\left(1+\frac{18^{d}}{1-2^{-(d+r)}}\right)m^{-r/d}=Cm^{-r/d}.
\]
Now for $m\in\N$ and $A\in\mathcal{A}_{m}$ arbitrary, there find
$n\in\N$ with $2^{nd}<m\leq2^{\left(n+1\right)d}$. Using our result
on the special subsequence, we get 
\[
\int d\left(x,A\right)^{r}\d\Lambda\left(x\right)\leq C2^{-(n+1)r}=\left(C2^{-r}\right)m^{-r/d},
\]
which proves our claim by setting $C_{r,d}:=2^{-r}C$.
\end{proof}
\begin{lem}
\label{lem:QC_UnitCube} For the uniform distribution $\Lambda$ on
$\Q$, the $d$-dimensional quantization coefficient $\c_{r,d}\left(\Lambda\right)$
of order exists, is finite, and positive only for $r\in\left(-d,+\infty\right)$.
\end{lem}

\begin{proof}
This statement is well known for $r\geq1$ (see e.\,g\@. \cite{MR1764176};
the proof there also works for $r\in\left(0,1\right]$) and for $r=0$
see \cite[Theorem 3.2]{MR2055056}. Now for $r\in\left(-d,0\right)$,
note that by an observation from the introduction of \cite{KN22b}
we know that for any subcube $Q\subset\Q$ with side length $a\in\left(0,1\right)$
and $\Lambda_{Q}$ denoting the normalized restriction of $\Lambda$
to $Q$, we have
\[
V_{n,r}\left(\Lambda_{Q}\right)=a^{r}V_{n,r}\left(\Lambda\right)
\]
Now, for $k\in\N$, let us divide $\Q$ evenly into $k^{d}$ axis-parallel
subcubes $\left\{ Q_{i,k}:i=1,\ldots,k^{d}\right\} $ with side length
$1/k$. Then we have $\Lambda=\sum k^{-d}\Lambda_{Q_{i,k}}$ and hence
by part (2) of \prettyref{lem:Basic_convex}, for $m\in\N$ and $n\coloneqq k^{d}m,$
\[
V_{n,r}\left(\Lambda\right)\geq\sum_{i=1}^{k^{d}}k^{-d}V_{m,r}\left(\Lambda_{Q_{i,k}}\right)=k^{-r}V_{m,r}\left(\Lambda\right).
\]
Therefore,
\[
n^{r/d}V_{n,r}\left(\Lambda\right)=k^{r}m^{r/d}V_{n,r}\left(\Lambda\right)\geq m^{r/d}V_{m,r}\left(\Lambda\right).
\]
This gives for each $m\in\N$
\[
\limsup_{n\to\infty}n^{1/d}\er_{n,r}\left(\Lambda\right)\leq m^{1/d}\er_{m,r}\left(\Lambda\right).
\]
Hence, in the chain of inequalities 
\[
\limsup_{n\to\infty}n^{1/d}\er_{n,r}\left(\Lambda\right)\leq\inf_{m}m^{1/d}\er_{m,r}\left(\Lambda\right)\leq\liminf_{n\to\infty}n^{1/d}\er_{n,r}\left(\Lambda\right)
\]
in fact equality holds and the claimed limit exists and is smaller
than $V_{1,r}\left(\Lambda\right)^{1/r}<\infty$. The limit is also
positive by \prettyref{lem:upper_bound_QE_Lebesgue}. The fact that
$\c_{r,d}\left(\Lambda\right)=0$ for $r\leq-d$ follows from \prettyref{prop:r<-dimoo(nu)}.
\end{proof}
\begin{proof}[Proof of \prettyref{thm:absoltutelyContinuousCase}]
 With $\nu\coloneqq h\Lambda$, the statement for positive $r$ follows
by \cite[Theorem 6.2]{MR1764176} and $r<-\dim_{\infty}(\nu)$ is
also clear from \prettyref{prop:r<-dimoo(nu)}. For $r\in(-d,0)$,
we first follow almost literally the proof of \cite[Theorem 6.2]{MR1764176}:

First, let us consider only densities $h$ that are constant on cubes
from $\mathcal{D}_{n}$. By using \prettyref{lem:Basic_convex} and
\prettyref{lem:QC_UnitCube} at the appropriate places and exchanging
all relevant inequalities with their inverses and limes superior with
limes inferior we see that the theorem holds for such densities. For
example, for $r$ negative \cite[Lemma 6.8]{MR1764176} provides a
unique maximizer (instead of minimizer for positive $r$) as required
in the proof for the upper bound of $\limsup_{n}n^{r/d}V_{n,r}\left(h\Lambda\right)$.
Namely, at the appropriate place we need the following general observation,
which is an immediate consequence of Hölder's inequality: For $m\in\mathbb{N}$
and numbers $s_{i}>0$, let $B=\left\{ (v_{1},\ldots,v_{m})\in(0,\infty)^{m}:\sum_{i=1}^{m}v_{i}\leq1\right\} $
and 
\[
t_{i}=\frac{s_{i}^{d/(d+r)}}{\sum_{j=1}^{m}s_{j}^{d/(d+r)}},\quad1\leq i\leq m.
\]
Then the function $F:B\to\mathbb{R}_{+},\quad F(v_{1},\ldots,v_{m})=\sum_{i=1}^{m}s_{i}v_{i}^{-r/d}$satisfies
\[
F(t_{1},\ldots,t_{m})=\left(\sum_{i=1}^{m}s_{i}^{d/(d+r)}\right)^{(d+r)/d}=\max_{(v_{1},\ldots,v_{m})\in B}F(v_{1},\ldots,v_{m})
\]
and $(t_{1},\ldots,t_{m})$ is the unique maximizer of $F$.

For all $r\in\left(d/s_{h}-d,0\right)$ we have $\left\Vert h\right\Vert _{\frac{d}{d+r}}<\infty$
and, by Jensen's inequality, we find that for the conditional expectation
$h_{k}\coloneqq\mathbb{E}\left(h\mid\mathcal{D}_{k}\right)$ we have
$\int h_{k}^{d/(d+r)}\d\Lambda\leq\int h^{d/(d+r)}\d\Lambda<\infty$.
Since $\sigma(\mathcal{D}_{k})\nearrow\sigma(\mathcal{D})=\mathcal{B}$,
by the martingale convergence theorem, we infer that the sequence
$\left(h_{k}\right)$ converges to $h$ almost surely and, for all
$q\leq d/\left(d+r\right)$, also in $L_{\Lambda}^{q}$; in particular,
we have $\left\Vert h_{k}\right\Vert _{q}\to\left\Vert h\right\Vert _{q}$.

For the final step in the proof we argue as follows. For $d/s_{h}-d<r'<r<0$,
$n\in\N$, and $A\in\mathcal{A}_{n}$ we have by Hölder's inequality
\begin{align*}
\int d\left(\,\cdot\,,A\right)^{r}h\d\Lambda & =\int d\left(\,\cdot\,,A\right)^{r}\left(h-h_{k}+h_{k}\right)\d\Lambda\\
 & \leq V_{n,r}\left(h_{k}\Lambda\right)+\int d\left(\,\cdot\,,A\right)^{r}\left|h-h_{k}\right|\d\Lambda\\
 & \leq V_{n,r}\left(h_{k}\Lambda\right)+\left(\int d\left(\,\cdot\,,A\right)^{-dr/r'}\d\Lambda\right)^{-r'/d}\left\Vert h-h_{k}\right\Vert _{d/\left(d+r'\right)}\\
 & \leq V_{n,r}\left(h_{k}\Lambda\right)+\left(V_{n,-rd/r'}\left(\Lambda\right)\right)^{-r'/d}\left\Vert h-h_{k}\right\Vert _{d/\left(d+r'\right)}.
\end{align*}
If we take the supremum over all $A\in\mathcal{A}_{n}$, multiply
by $n^{r/d}$ and take the limes superior, we obtain, 
\begin{align*}
\limsup_{n\to\infty}n^{r/d}V_{n,r}\left(h\Lambda\right) & \leq\limsup_{n\to\infty}n^{r/d}V_{n,r}\left(h_{k}\Lambda\right)+\limsup_{n\to\infty}n^{r/d}V_{n,-rd/r'}\left(\Lambda\right)^{-r'/d}\left\Vert h-h_{k}\right\Vert _{d/\left(d+r'\right)}\\
 & =\left(\Phi_{r}\left(h_{k}\right)\c_{r,d}\left(\Lambda\right)\right)^{r}+\c_{rd/r',d}\left(\Lambda\right)^{-r'/d}\left\Vert h-h_{k}\right\Vert _{d/\left(d+r'\right)}\\
 & \to\left(\Phi_{r}\left(h_{k}\right)\c_{r,d}\left(\Lambda\right)\right)^{r}\,\,\:\,\text{for }k\to\infty.
\end{align*}
Taking the $r$-th root, this proves the lower bound on $\liminf_{n\to\infty}n^{1/d}\er_{n,r}\left(h\Lambda\right)$.

Similarly,
\begin{align*}
V_{n,r}\left(h\Lambda\right) & \geq\int d\left(\,\cdot\,,A\right)^{r}h\d\Lambda=\int d\left(\,\cdot\,,A\right)^{r}\left(h-h_{k}+h_{k}\right)\d\Lambda\\
 & \geq\int d\left(\,\cdot\,,A\right)^{r}h_{k}\d\Lambda-\int d\left(\,\cdot\,,A\right)^{r}\left|h-h_{k}\right|\d\Lambda\\
 & \geq\int d\left(\,\cdot\,,A\right)^{r}h_{k}\d\Lambda-\left(\int d\left(\,\cdot\,,A\right)^{-dr/r'}\d\Lambda\right)^{-r'/d}\left\Vert h-h_{k}\right\Vert _{d/\left(d+r'\right)}\\
 & \geq\int d\left(\,\cdot\,,A\right)^{r}h_{k}\d\Lambda-\left(V_{n,-rd/r'}\left(\Lambda\right)\right)^{-r'/d}\left\Vert h-h_{k}\right\Vert _{d/\left(d+r'\right)}.
\end{align*}
Again, if we take the supremum over all $A\in\mathcal{A}_{n}$ —this
time on the right-hand side—, multiply by $n^{r/d}$ and take the
limes inferior, we obtain,
\begin{align*}
\liminf_{n\to\infty}n^{r/d}V_{n,r}\left(h\Lambda\right) & \geq\liminf_{n\to\infty}n^{r/d}V_{n,r}\left(h_{k}\Lambda\right)-\limsup_{n\to\infty}n^{r/d}V_{n,-rd/r'}\left(\Lambda\right)^{-r'/d}\left\Vert h-h_{k}\right\Vert _{d/\left(d+r'\right)}\\
 & =\left(\Phi_{r}\left(h_{k}\right)\c_{r,d}\left(\Lambda\right)\right)^{r}-\c_{rd/r',d}\left(\Lambda\right)^{-r/d'}\left\Vert h-h_{k}\right\Vert _{d/\left(d+r'\right)}\\
 & \to\left(\Phi_{r}\left(h\right)\c_{r,d}\left(\Lambda\right)\right)^{r}\,\,\:\,\text{for }k\to\infty.
\end{align*}
This proves the upper bound on $\limsup_{n\to\infty}n^{1/d}\er_{n,r}\left(h\Lambda\right)$.

To cover the case $r=d/s_{h}-d$ with $\left\Vert h\right\Vert _{s_{h}}=\infty$,
we consider the truncated version $h\wedge s\1$, for $s>0$, which
gives $\limsup n^{1/d}\er_{n,r}\left(h\Lambda\right)\leq\limsup n^{1/d}\er_{n,r}\left(\left(h\wedge s\1\right)\Lambda\right)=\c_{r,d}\left(\Lambda\right)\Phi_{r}\left(h\wedge s\right)\to0$
for $s\nearrow\infty$.

For $r=0$ the claimed upper bound is contained in \cite[Theorem 3.4]{MR2055056}.
For the lower bound we make use of our results for negative $r$ together
with the fact that $r\mapsto n^{1/d}\left(V_{n,r}\left(h\Lambda\right)\right)^{1/r}$,
$n\in\N$, is monotonically increasing on $r\in\left(-d,0\right]$
and that
\begin{align*}
\lim_{r\nearrow0}\Phi_{r}\left(h\right) & =\lim_{r\nearrow0}\left(\int h^{-r/\left(d+r\right)}h\d\Lambda\right)^{\frac{-\left(d+r\right)/r}{-d}}\\
 & =\lim_{t\searrow0}\exp\left(-1/d\cdot1/t\cdot\log\int\exp\left(t\log\left(h\right)\right)h\d\Lambda\right)\\
 & =\exp\frac{\int\frac{d}{dt}\exp\left(t\log\left(h\right)\right)|_{t=0}h\d\Lambda}{-d\int\exp\left(0\log\left(h\right)\right)h\d\Lambda}\\
 & =\exp-\left(1/d\right)\int\exp\left(0\log\left(h\right)\right)\log\left(h\right)h\d\Lambda=\Phi_{0}\left(h\right).
\end{align*}
In the third equality we used that $\left(\exp\left(t\log\left(h\right)\right)h\right)/t<h^{1+\epsilon}\in L_{\Lambda}^{1}$
for all $t\in\left(0,\epsilon\right)$ and some $\epsilon>0$ in tandem
with Lebesgue's dominated convergence theorem.
\end{proof}
\begin{example}
\label{exa:abs_cnt_with_strict_inequ} We provide an example of an
absolutely continuous measure $\nu$ on $\Q\coloneqq[0,1]$ such that
$-d<-\dim_{\infty}\left(\nu\right)<d/s_{h}-d$. For this we consider
a disjoint family $\left(I_{n,k}:n\in\N,1\leq k\leq2^{n}\right)$
of pairwise disjoint subintervals of $\left[0,1\right]$ such that
for each $n\in\N$ and $k\in\left\{ 1,\ldots,2^{n}\right\} $ we have
$\left|I_{n,k}\right|=2^{-3n+1}$. We define a measure by
\[
\nu\coloneqq\sum_{n\in\N,k=1,\dots,2^{n}}\Lambda\left(I_{n,k}\right)^{-1/2}\Lambda|_{I_{n,k}}.
\]
 Since $\nu\left(I_{n,k}\right)=\Lambda\left(I_{n,k}\right)^{1/2}$
and 
\[
\nu\left(\left[0,1\right]\right)=\sum_{n\in\N,k=1,\dots,2^{n}}\Lambda\left(I_{n,k}\right)^{1/2}=\sum_{n\in\N}2^{\left(-3n+1\right)/2+n}<\infty
\]
this measure is finite and absolutely continuous with density 
\[
h\coloneqq\sum_{n\in\N,k=1,\dots,2^{n}}\Lambda\left(I_{n,k}\right)^{-1/2}\1_{I_{n,k}}.
\]
Then $\int_{\bigcup_{k}I_{n,k}}h^{s}\d\Lambda=2^{n}2^{s\left(n3/2-1/2\right)-3n+1}=2^{n\left(s3/2-2\right)+1-s/2}$
and therefore 
\[
\int h^{s}\d\Lambda=\sum_{n\in\N}2^{n\left(s3/2-2\right)+1-s/2}=2^{-s/2+1}\sum_{n\in\N}2^{\left(3/2s-2\right)n}\begin{cases}
=\infty, & \text{for }s\geq4/3,\\
<\infty, & \text{for }s<4/3.
\end{cases}
\]
Hence, $s_{h}=4/3$. On the other hand, $\dim_{\infty}\left(\nu\right)=1/2$
and $d=1$ giving 
\[
-d=-1<-\dim_{\infty}\left(\nu\right)=-1/2<-1/4=d/s_{h}-d.
\]
\end{example}

\begin{example}
\label{exa:Our-second-example}Our second example concerns an absolutely
continuous measure $\nu$ on $\Q\coloneqq\left[0,1\right]$ this time
with density 
\[
h(x)=x^{-1/2}\left(\log\left(\frac{x}{10}\right)\right)^{-2},\;x\in\left[0,1\right].
\]
A straightforward calculation shows that $\dim_{\infty}(\nu)=1/2$
and $s_{h}=2$. This fact in combination with \prettyref{lem:AbsolLq}
and \prettyref{lem:lowerboudonBeta} gives 
\[
\beta_{\nu}\left(q\right)=\begin{cases}
1-q, & 0\leq q\leq2,\\
-\left(1/2\right)q, & q>2,
\end{cases}
\]
and consequently, $\tau_{\J_{\nu,r}}(q)=\beta_{\nu}(q)-rq$ for $q\in[0,2]$,
$a_{\nu}=2$, and $D_{t}(\nu)=1$ for all $0<t<-1/2$. In particular,
by right continuity (\prettyref{lem:GeometricBounds}), monotonicity
(\prettyref{lem:monontonicity}) and regularity (\prettyref{thm:LqRegularImpliesRegular})
we have
\[
\lim_{t\searrow-1/2}D_{t}(\nu)=1=\frac{a_{\nu}}{a_{\nu}-1}\dim_{\infty}(\nu)\geq D_{-1/2}(\nu).
\]
 It is more involved to show that $D_{-1/2}(\nu)\geq1$:

For $I(A)\coloneqq\int_{0}^{1}\left(d(x,A)\right)^{-1/2}h(x)\d\Lambda\left(x\right)$
with $A=\{0=a_{1}<a_{2}<\cdots<a_{n}=1\}\subset[0,1]$, $n\geq2$,
our claim is that for a universal constant $C>0$,
\[
I(A)\leq C\sqrt{n}.
\]
This implies that $\sup_{A\in\mathcal{A}_{n}}V_{n,-1/2}(\nu)\leq C\sqrt{n}$
and therefore $\underline{D}_{-1/2}(\nu)\geq1$. Note that we have
made use of the observation that the assumption $\left\{ 0,1\right\} \subset A$
does not result in a loss of generality as a consequence of $n\mapsto V_{n,-1/2}(\nu)$
being monotonically increasing. We proceed to prove this claim as
follows: \textbf{}\\
\textbf{~}\\
\emph{Partitioning the interval.} For each $1\leq i\leq n-1$, set
$J_{i}:=\left[a_{i},a_{i+1}\right)$, a partition of $\left[0,1\right)$,
and define $\ell_{i}:=a_{i+1}-a_{i}$. Then, $I(A)=\sum_{i=1}^{n-1}I_{i}$
with $I_{i}\coloneqq\int_{J_{i}}\left(d(x,A)\right)^{-1/2}h(x)\d\Lambda\left(x\right).$
We derive two useful upper bounds for this last integral: Since on
each cell $J_{i}$, the density $h$ is strictly positive and decreasing,
we have
\begin{equation}
I_{i}=\int\limits_{a_{i}}^{\frac{a_{i}+a_{i+1}}{2}}\negthickspace\negthickspace\frac{h\left(x\right)}{\sqrt{x-a_{i}}}\d\Lambda\left(x\right)+\int\limits_{\frac{a_{i}+a_{i+1}}{2}}^{a_{i+1}}\negthickspace\negthickspace\frac{h(x)}{\sqrt{a_{i+1}-x}}\d\Lambda\left(x\right)\leq2h(a_{i})\sqrt{2}\ell_{i}^{1/2}\leq3h\left(a_{i}\right)\ell_{i}^{1/2}.\label{eq:ExUpperBound1}
\end{equation}
The second bound considers the behavior for $a_{i}$ close to $0$,
for which we get
\begin{equation}
I_{i}\le2\int_{0}^{\ell_{i}/2}\frac{1}{x(\log(x/10))^{2}}\d\Lambda\left(x\right)=\frac{2}{|\log(\ell_{i}/20)|}.\label{eq:ExUpperBound2}
\end{equation}
\emph{Partition of the indices.} We now partition the set of indices
into three disjoint cases according to the location and size of the
intervals:
\[
R:=\left\{ i:a_{i}<\e^{-\sqrt{n}}\right\} ,\qquad S:=\left\{ i:\ell_{i}>a_{i}\right\} \setminus R,\qquad T:=\left\{ i:\ell_{i}\leq a_{i}\right\} \setminus R.
\]
 For each $P\in\{R,S,T\}$, define $I_{P}:=\sum_{i\in P}I_{i}$. Thus,
$I(A)=I_{R}+I_{S}+I_{T}.$
\begin{lyxlist}{00.00.0000}
\item [{\emph{Case~1:~Small~$a_{i}\,\left(i\in R\right)$.}}] Using
bound \prettyref{eq:ExUpperBound2} above, and that for $a_{i}<e^{-\sqrt{n}}$,
we have $\ell_{i}\leq\e^{-\sqrt{n}}$ by construction, so $|\log(\ell_{i}/20)|\geq\sqrt{n}$
for large $n$. Therefore, $I_{R}\leq n\cdot\max_{i\in R}\frac{2}{|\log(\ell_{i}/20)|}\leq2n/\sqrt{n}=2\sqrt{n}.$
\item [{\emph{Case~2:~Large~intervals~$\left(i\in S\right)$.}}] For
intervals where $\ell_{i}>a_{i}$, using bound \prettyref{eq:ExUpperBound2}
again, $\int_{J_{i}}h(x)|x-a_{i}|^{-1/2}\d\Lambda\left(x\right)\leq1,$
and the number of such intervals is controlled by $2^{\card\left(S\right)-1}\e^{-\sqrt{n}}\leq1$,
so $\card\left(S\right)\leq3\sqrt{n}$, hence $I_{S}\leq3\sqrt{n}.$
\item [{\emph{Case~3:~Small~intervals~away~from~0~$\left(i\in T\right)$.}}] For
these, applying Hölder's inequality, the bound \prettyref{eq:ExUpperBound1}
and using the integral comparison criterion,
\begin{align*}
I_{T} & \leq3\sum_{i\in T}h\left(a_{i}\right)\sqrt{\ell_{i}}\leq3\sqrt{\card\left(T\right)}\left(\sum_{i\in T}\left(h\left(a_{i}\right)\right)^{2}\ell_{i}\right)^{1/2}\\
 & \leq3\sqrt{n}\left(\sum_{i\in T}2\left(h\left(a_{i}+\ell_{i}\right)\right)^{2}\ell_{i}\right)^{1/2}\leq6\sqrt{n}\underbrace{\left(\int h^{2}\d\Lambda\right)^{1/2}}_{=1/\sqrt{3}\cdot\left(\log\left(10\right)\right)^{-3/2}}<\sqrt{n}.
\end{align*}
\end{lyxlist}
Combining the above, we can conclude that the constant $C$ can be
chosen as $6$.
\end{example}

\printbibliography

\end{document}